\documentclass[a4paper,12pt]{amsart}
\usepackage[utf8]{inputenc}
\usepackage{csquotes}
\usepackage{amsmath, amssymb, amsthm, xcolor, enumerate, enumitem}
\usepackage[english]{babel}
\usepackage{tikz-cd}
\usepackage[all]{xy}
\usepackage{hyperref}
\usepackage{microtype}
\usepackage{graphicx}
\usepackage{comment} 
\usepackage{mathtools}

\newcounter{theorem}
\newtheorem{theorem}[theorem]{Theorem}
\newtheorem{lemma}[theorem]{Lemma}
\newtheorem{prop}[theorem]{Proposition}

\theoremstyle{definition}
\newtheorem{defn}[theorem]{Definition}

\theoremstyle{remark}
\newtheorem*{remark*}{Remark}
\newtheorem{rmk}[theorem]{Remark}

\numberwithin{equation}{section}
\allowdisplaybreaks

\newcommand{\C}{\mathrm{C}^*}
\newcommand{\Z}{\mathcal{Z}}
\newcommand{\id}{\mathrm{id}}

\newcommand{\N}{\mathbb{N}}
\newcommand{\Aut}{\mathrm{Aut}}
\newcommand{\Ad}{\mathrm{Ad}}

\newcommand{\ev}{\mathrm{ev}}
\newcommand{\Aff}{\mathrm{Aff}}
\newcommand{\Ell}{\mathrm{Ell}}
\newcommand{\UCT}{\mathrm{UCT}}

\newcommand{\KMS}{\mathrm{KMS}}
\newcommand{\supp}{\mathrm{supp}}

\title{The admissible KMS bundles on classifiable C$^*$-algebras}

\author{Robert Neagu}
\address{\hskip-\parindent Robert Neagu, Department of mathematics, KU Leuven, Celestijnenlaan 200B, 3001, Leuven, Belgium.}
\email{robert.neagu@kuleuven.be}

\thanks{Funded by the European Union. Views and opinions expressed are however those of the authors only and do not necessarily reflect those of the European Union or the European Research Council. Neither the EU nor the ERC can be held responsible for them. For the purpose of open access, the author has applied a CC-BY license to any author accepted manuscript arising from this submission.}

\begin{document}
\maketitle

\begin{abstract}
Given any unital, finite, classifiable $\C$-algebra $A$ with real rank zero and any compact simplex bundle with the fibre at zero being homeomorphic to the space of tracial states on $A$, we show that there exists a flow on $A$ realising this simplex. Moreover, we show that given any unital $\UCT$ Kirchberg algebra $A$ and any proper simplex bundle with empty fibre at zero, there exists a flow on $A$ realising this simplex.
\end{abstract}

\numberwithin{theorem}{section}	

\section*{Introduction}
\renewcommand*{\thetheorem}{\Alph{theorem}}

The study of group actions is ubiquitous within the subject of operator algebras, often providing deep structural properties.\ While classifying actions of discrete groups proved to be an indispensable tool for understanding the structure and symmetries of operator algebras, certain applications to geometry or physics are usually related to time evolutions, or continuous actions of $\mathbb{R}$. In the setting of von Neumann algebras, such actions appeared most prominently in the form of modular automorphism groups through the theory of Tomita and Takesaki (\cite{Tak70}).\ One convenient feature of the Tomita-Takesaki theory is that one can check if a certain flow is a modular flow by verifying the so-called $\KMS$ condition.\ Precisely, a faithful normal state $\varphi$ on a von Neumann algebra satisfies the $\KMS$ conditions for a flow if and only if the given flow is the modular flow induced by $\varphi$. Therefore, investigating the set of corresponding $\KMS$ states sheds light on questions regarding structure and classification of continuous actions of $\mathbb R$ on von Neumann algebras.

On the $\C$-algebraic side, named after mathematical physicists Kubo, Martin, and Schwinger, $\KMS$ states are a special class of states on any $\C$-algebra admitting a continuous action of $\mathbb R$. However, the collection of $\KMS$ states for a given flow on a $\C$-algebra can be quite intricate.\ Indeed, given a flow on a unital separable $\C$-algebra, the collection of $\KMS$ states corresponding to that flow, also called the $\KMS$ bundle of the flow, has the structure of a \emph{proper simplex bundle} (see for example \cite[Section 2.1]{ET}).

The quest of constructing flows inducing a given $\KMS$ simplex has its roots in work of Bratteli, Elliott, Herman, and Kishimoto in \cite{BEH80,BEK80,BEK86}, where flows were constructed on certain classes of simple $\C$-algebras.\ More recently, extensive work has been done to examine possible $\KMS$ bundles on various special classes of $\C$-algebras (see for example \cite{HLRS15,T17,ALN20,CV22,CT22,C23,BE23}).\ Moreover, in \cite{KT21,ET} Elliott and Thomsen built flows on simple AF-algebras with prescribed $\KMS$ behaviour.\ These results were then extended by Elliott, Sato, and Thomsen in \cite{EST,ES} to any unital classifiable $\C$-algebra with a unique trace, as well as any unital $\UCT$ Kirchberg algebra with no torsion in the $K_1$-group.

Using the new abstract classification of morphisms from \cite{classif}, we will prove a similar result for any unital tracial classifiable $\C$-algebra with real rank zero. However, as in \cite[Theorem 3.14]{EST}, the assumption that the $\KMS$ bundle is compact is still needed in the case of tracial $\C$-algebras.\ In \cite{ES}, this assumption was removed in the case when the given classifiable $\C$-algebra has a unique trace. Moreover, without any restrictions on the given $\KMS$ bundle, we will remove the assumption of no torsion in the $K_1$-group from \cite[Theorem 5.1]{EST} to obtain the definitive result for unital $\UCT$ Kirchberg algebras. The main results of this paper are the following.

\begin{theorem}\label{thm:ExistenceKirchberg}
Let $A$ be a unital $\UCT$ Kirchberg algebra and let $(S,\pi)$ be a proper simplex bundle such that $\pi^{-1}(0)=\emptyset$. Then there exists a $2\pi$-periodic flow $\theta$ on $A$ such that its induced $\KMS$ bundle $(S^\theta,\pi^\theta)$ is isomorphic to $(S,\pi)$.  
\end{theorem}

\begin{theorem}\label{thm:ExistenceTracial}
Let $A$ be a unital, stably finite, classifiable $\C$-algebra with real rank zero and let $(S,\pi)$ be a compact simplex bundle such that $\pi^{-1}(0)$ is affinely homeomorphic to $T(A)$.\ Then there exists a $2\pi$-periodic flow $\theta$ on $A$ such that its induced $\KMS$ bundle $(S^\theta,\pi^\theta)$ is isomorphic to $(S,\pi)$. 
\end{theorem}

The proof of Theorem \ref{thm:ExistenceKirchberg} follows the skeleton provided by \cite[Theorem 5.1]{EST}. Essentially, given a unital $\UCT$ Kirchberg algebra $A$ and a suitable simplex bundle $(S,\pi)$, one can build a pair of abelian groups containing both information from the $K$-theory of $A$ as well as the bundle $(S,\pi)$. Then, using the classification results in \cite{elliottclassifRR0}, we can obtain a stable $A\mathbb{T}$-algebra of real rank zero $B$ realising the constructed pair of abelian groups. Then, $A$ can be identified with a full corner of a crossed product of $B$ by $\mathbb{Z}$. It is precisely the fact that $B$ is an $A\mathbb{T}$-algebra rather than an AF-algebra (see \cite[Theorem 5.1]{EST}) that allows us to remove the assumption that $A$ has no torsion in $K_1$. 

The proof of Theorem \ref{thm:ExistenceTracial} uses a similar strategy.\ Given a unital classifiable $\C$-algebra $A$ with real rank zero, we realise it as a full corner of a crossed product of a stabilised unital classifiable $\C$-algebra $B\otimes\mathcal{K}$ by $\mathbb{Z}$. A key ingredient in obtaining an appropriate automorphism of $B\otimes\mathcal{K}$ is the classification of unital $^*$-homomorphisms from \cite{classif}. Moreover, compared to \cite[Theorem 3.14]{EST}, where the corresponding result was proved for the Jiang-Su algebra $\Z$, extra care has to be taken in checking that $A$ and the corner of the crossed product built above have the same pairing between $K$-theory and traces. The assumption that $A$ has real rank zero is instrumental in constructing an Elliott invariant of a classifiable $\C$-algebra $B$. Precisely, one can build an ordered abelian group $G$ and then realise the trace space of $B$ as the space of states on $G$. However, to make sure that the state space on $G$ is a metrisable Choquet simplex, the fact that $A$ has real rank zero is heavily used. The assumption that the space $S$ is compact in Theorem \ref{thm:ExistenceTracial} seems to be more difficult to remove. In particular, it ensures that the constructed ordered abelian group is simple, and hence the $\C$-algebra $B$ is simple. Then, one can use a classification of automorphisms of $B\otimes\mathcal{K}$ to obtain a suitable automorphism $\gamma$. Note that a definitive classification theorem for automorphisms of general nonsimple $\Z$-stable $\C$-algebras is currently out of reach, so the compactness assumption appears necessary at this point. 

This paper is organised as follows.\ In Section \ref{sect: Prelim} we collect some definitions regarding $\KMS$ states and classification invariants.\ Then, Section \ref{sect: ProofA} is concerned with the proof of Theorem \ref{thm:ExistenceKirchberg}, while  Section \ref{sect: ProofB} will focus on the proof of Theorem \ref{thm:ExistenceTracial}.

\subsection*{Acknowledgements}
The author was partly supported by the EPSRC grant EP/R513295/1 and by the European Research Council under the European Union's Horizon Europe research and innovation programme (ERC grant AMEN-101124789).
The author would like to thank Stuart White and Jamie Gabe for their supervision on this project and Stuart White for numerous helpful comments on earlier drafts of this paper. The author would also like to thank Hung-Chang Liao for allowing him to include a proof of Lemma \ref{lemma: TracesCrossProd}. This article will form part of the author's DPhil thesis.

\allowdisplaybreaks

\section{Preliminaries}\label{sect: Prelim}
\numberwithin{theorem}{section}

\subsection{KMS states}\label{sect: KMS}

In this subsection, we record some facts about $\KMS$ states. We refer the reader to \cite{ThomsenBook} for a detailed account of the topic. 

A \emph{flow} $\theta$ on a $\C$-algebra $A$ is a continuous one-parameter group of automorphisms $(\theta_t)_{t\in\mathbb{R}}$.

\begin{defn}\label{defn: KMS}
Let $A$ be a $\C$-algebra, $\theta$ be a flow on $A$, and $\beta\in\mathbb{R}$. Then, $\omega$ is said to be a \emph{$\beta$-$\KMS$ weight} for $\theta$ on $A$ if $\omega$ is a weight on $A$ such that $$\omega\circ\theta_t=\omega \ \text{for all} \ t\in\mathbb{R}$$ and $$\omega(a^*a)=\omega\left(\theta_{-\frac{i\beta}{2}}(a)\theta_{-\frac{i\beta}{2}}(a)^*\right), \ a\in A \ \text{analytic}.\footnote{An element $a \in A$ is analytic for $\theta$ if the map $t \in \mathbb R \mapsto \theta_t(a) \in A$ extends to an entire analytic map $\mathbb{C} \to A.$}$$ If, in addition, $\omega$ has norm $1$, then it is called a $\beta$-$\KMS$ state for $\theta$ on $A$.  
\end{defn}

\begin{rmk}
Note that a $0$-$\KMS$ weight for $\theta$ is a $\theta$-invariant trace. Moreover, a $0$-$\KMS$ state for $\theta$ is a $\theta$-invariant tracial state.
\end{rmk}

If $A$ is a separable unital $\C$-algebra and $\theta$ is a flow on $A$, we consider the collection of $\KMS$ states for $\theta$ on $A$. Following the notation in \cite[Section 2]{EST}, for each $\beta\in\mathbb{R}$, let $S_\beta^\theta$ denote the set of $\beta$-$\KMS$ states for $\theta$. If we denote the set of states on $A$ by $S(A)$, let $$S^\theta=\left\{(\omega,\beta)\in S(A)\times\mathbb{R}: \omega\in S_\beta^\theta\right\},$$ and equip it with the relative topology from $S(A)\times\mathbb{R}$. In particular, since $A$ is separable, the topology on $S^{\theta}$ is metrisable. Let $\pi^\theta:S^\theta\to\mathbb{R}$ be the projection onto the second coordinate. Then, the pair $(S^\theta, \pi^\theta)$ will be called the \emph{$\KMS$ bundle} of $\theta$.\ This was characterised abstractly in \cite{EST} as we review below.

Fix a second countable locally compact Hausdorff space $S$ and let $\pi$  be a continuous map from $S$ to $\mathbb{R}$. The pair $(S, \pi)$ is said to be a \emph{simplex bundle} if the inverse image $\pi^{-1}(t)$ is a compact metrisable Choquet simplex in the relative topology from $S$ for any $t\in \mathbb{R}$. For such pairs $(S,\pi)$, we denote by $\mathcal A(S, \pi)$ the set of continuous functions $f$ from $S$ to $\mathbb R$ such that 
the restriction $f|_{\pi^{-1}(t)}$ of $f$ to $\pi^{-1}(t)$ is affine for any $t\in\mathbb R$.
\begin{defn}{\cite[Definition 2.1]{EST}}\label{defn: SimplexBundle}
A simplex bundle $(S, \pi)$ is \emph{proper}, if
\begin{itemize}
\item $\pi^{-1}(K)$ is compact in $S$ for any compact subset $K$ of $\mathbb R$, 
\item for any $x\neq y$ in $S$, there exists  $f \in\mathcal A(S,\pi)$ such that $f(x) \neq f(y)$.
\end{itemize} Moreover, if $S$ is compact, then $(S,\pi)$ is called a \emph{compact simplex bundle}. Two proper simplex bundles $(S,\pi)$ and $(S',\pi')$ are \emph{isomorphic} if there exists a homeomorphism $\phi : S \to S'$ such that $$\pi' \circ \phi = \pi \ \text{and} \ \phi: \pi^{-1}(\beta) \to {\pi'}^{-1}(\beta) \ \text{is affine for all} \ \beta \in \mathbb R.$$
\end{defn}

\begin{rmk}\label{rmk: MetrisBundle}
By \cite[Theorem 9.2.2]{ThomsenBook}, the collection of $\KMS$ states for a flow on a unital separable $\C$-algebra is a proper simplex bundle. Hence, this is a necessary assumption for the pair $(S,\pi)$ to be a suitable invariant. Moreover, a proper simplex bundle $(S,\pi)$ is second countable, locally compact, Hausdorff, so the topology on $S$ is metrisable.
\end{rmk}

\begin{lemma}\label{lemma: BundleHomeom}
Let $\eta:(S',\pi')\to (S,\pi)$ be a bijection between proper simplex bundles such that $\pi\circ\eta=\pi'$. If either $\eta$ or $\eta^{-1}$ is continuous, then $\eta$ is a homeomorphism.
\end{lemma}

\begin{proof}
Suppose that $\eta$ is continuous. It remains to show that $\eta^{-1}$ is continuous. We show that $\eta$ maps closed sets to closed sets. By \cite[Corollary]{ProperMaps}, it suffices to check that $\eta$ is a proper map i.e. the preimage of any compact set is compact. Let $K\subseteq S$ be compact and note that $\pi(K)$ is compact since $\pi$ is continuous. Since $\pi\circ\eta=\pi'$, it follows that $$\eta^{-1}(K)\subseteq \eta^{-1}(\pi^{-1}(\pi(K)))=(\pi')^{-1}(\pi(K)).$$ But $(\pi')^{-1}(\pi(K))$ is compact since $(S',\pi')$ is a proper simplex bundle (see Definition \ref{defn: SimplexBundle}), so $\eta^{-1}(K)$ is compact as a closed subset of a compact set. This finishes the argument. The case when $\eta^{-1}$ is assumed to be continuous is completely symmetric.
\end{proof}

\subsection{Classification invariants}\label{sect: Invariants}

Let $A$ be a separable $\C$-algebra. We denote by $K_0(A)_+$ the set of positive elements in $K_0(A)$. We write $T(A)$ for the set of tracial states on $A$. Recall that for a unital $\C$-algebra $A$, there exists a natural pairing between the $K_0$-group and tracial states. For $n \in \mathbb N$ and $\tau \in T(A)$, we denote the canonical non-normalised extension of $\tau$ to a tracial functional on $M_n(A)$ by $\tau_n$. Then, the pairing map $\rho_A:K_0(A)\to \Aff(T(A))$ is given by $$\rho_A([p]_0 - [q]_0)(\tau) := \tau_n(p - q)$$ for all $\tau \in T (A)$ and all projections $p, q \in M_n(A)$.

\begin{defn}\label{defn: EllInv}
The \emph{Elliott invariant} of a unital separable $\C$-algebra $A$, denoted $\Ell(A)$, is given by 
\begin{equation*}
    \Ell(A)=\left(K_0(A), K_0(A)_+, [1_A]_0, K_1(A), T(A), \rho_A\right).
\end{equation*}
\end{defn}

\begin{rmk}
In the case of a stable $\C$-algebra, one needs to remove the class of the unit and consider the cone of densely defined lower semicontinuous traces instead of the set of tracial states.
\end{rmk}

Due to work by many mathematicians (e.g. \cite{kirchbergclass,phillipsclass,EGLN,TWW,GLN1,classif}), this invariant was shown to classify a large class of simple $\C$-algebras. In this paper, we will use the following classification theorem.

\begin{theorem}[see {\cite[Theorem A]{classif}}]\label{thm: Classif}
Unital simple separable nuclear $\Z$-stable $\C$-algebras satisfying Rosenberg and Schochet’s universal coefficient theorem (UCT) are classified by Elliott’s invariant consisting of $K$-theory and traces.    
\end{theorem}

\begin{rmk}\label{rmk: ElliottvsKTu}
On the class of $\C$-algebras in the statement of Theorem \ref{thm: Classif}, the order on the $K_0$-group is induced by traces.\footnote{This is a consequence of \cite[Proposition 4.13(ii)]{classif} which follows from \cite{Rordam04} (one can also see \cite[Theorem 1]{GJS00}).} Precisely, non-zero positive elements in $K_0$ are those elements that are strictly positive when evaluated on traces.\ Therefore, if $A$ is a $\C$-algebra as in the statement of Theorem \ref{thm: Classif}, then the invariant $\left(K_0(A), [1_A]_0, K_1(A), T(A), \rho_A\right)$ carries the same information as $\Ell(A)$.\ We will use this observation in the proof of Lemma \ref{lemma: IsomCorner}. 
\end{rmk}

Since we will need to check that certain corners of crossed products by $\mathbb Z$ satisfy the assumptions of the classification theorem, we collect below a number of permanence properties for the hypotheses in the statement of Theorem \ref{thm: Classif}.

\begin{prop}\label{prop: StableIsom}
Let $A,B$ be separable $\C$-algebras such that $A\otimes\mathcal{K}\cong B\otimes\mathcal{K}$. If $A$ is nuclear, $\Z$-stable, and satisfies the UCT, then B is nuclear, $\Z$-stable, and satisfies the UCT. 
\end{prop}

\begin{proof}
We first note that $B\otimes\mathcal{K}\cong A\otimes\mathcal{K}$ is nuclear and $\Z$-stable. Since nuclearity passes to hereditary subalgebras, $B$ is nuclear. Moreover, $B$ is $\Z$-stable by \cite[Corollary 3.1]{TW07}. Finally, the UCT is closed under stable isomorphisms (\cite[22.3.5 (a)]{blackadar}), so $B$ satisfies the UCT.
\end{proof}

\begin{prop}\label{prop: PermanenceProp}
Let $B$ be a separable $\C$-algebra, $\gamma$ be an automorphism of $B$, and $e\in B$ be a full projection in $B\rtimes_\gamma\mathbb{Z}$. Then the following hold. 

\begin{enumerate}[label=\textit{(\roman*)}]
\item The corner $e\left(B\rtimes_\gamma\mathbb{Z}\right)e$ is separable and unital.\label{item: PP1}
\item If $B\rtimes_\gamma\mathbb{Z}$ is simple, then $e\left(B\rtimes_\gamma\mathbb{Z}\right)e$ is simple.\label{item: PP2}
\item If $B$ is nuclear, then $e\left(B\rtimes_\gamma\mathbb{Z}\right)e$ is nuclear.\label{item: PP3}
\item If $B$ satisfies the UCT, then so does $e\left(B\rtimes_\gamma\mathbb{Z}\right)e$.\label{item: PP4}
\item If $B\rtimes_\gamma\mathbb Z$ is $\Z$-stable, then so is $e\left(B\rtimes_\gamma\mathbb{Z}\right)e$.\label{item: PP5}
\end{enumerate}
\end{prop}

\begin{proof}
Note that $e\left(B\rtimes_\gamma\mathbb{Z}\right)e$ is unital as $e$ is a projection, and it is separable since $B$ is separable and $\mathbb{Z}$ is a countable discrete group. If $B\rtimes_\gamma\mathbb{Z}$ is simple and $e$ is full, then $e\left(B\rtimes_\gamma\mathbb{Z}\right)e$ is simple. If $B$ is nuclear, then $B\rtimes_\gamma\mathbb{Z}$ is nuclear by \cite[Theorem 4.2.4]{BO}. Since nuclearity passes to hereditary subalgebras, $e\left(B\rtimes_\gamma\mathbb{Z}\right)e$ is nuclear. Suppose now that $B$ satisfies the UCT. Since the $\UCT$ is closed under crossed products by $\mathbb{Z}$ (\cite[22.3.5 (g)]{blackadar}), $B\rtimes_\gamma\mathbb{Z}$ satisfies the $\UCT$. Furthermore, $B\rtimes_\gamma\mathbb{Z}$ and $e\left(B\rtimes_\gamma\mathbb{Z}\right)e$ are stably isomorphic (see \cite{Brown77}) and the $\UCT$ is closed under stable isomorphisms (\cite[22.3.5 (a)]{blackadar}). Hence,  $e\left(B\rtimes_\gamma\mathbb{Z}\right)e$ satisfies the $\UCT$. Since $\Z$-stability passes to hereditary subalgebras (\cite[Corollary 3.1]{TW07}) and $B\rtimes_\gamma\mathbb{Z}$ is $\Z$-stable, it follows that $e\left(B\rtimes_\gamma\mathbb{Z}\right)e$ is $\Z$-stable. 
\end{proof}

Furthermore, Theorem \ref{thm: Classif} is complemented by a range-of-the-invariant result from \cite{InvRangeElliott}. Before stating the result, we need to record a few notions regarding ordered abelian groups.

\subsection{Ordered groups and the range of the invariant}

We refer the reader to \cite{goodearlbook} for a detailed account on the content of this subsection.

An ordered abelian group $(G,G_+)$ is called \emph{unperforated} if for any $n\in\mathbb{N}$ and any $g\in G$ such that $ng\in G_+$, then $g\in G_+$.\ Moreover, an ordered abelian group $(G,G_+)$ is called \emph{weakly unperforated} if for any $n\in\mathbb{N}$ and any $g\in G$ such that $ng\in G_+\setminus \{0\}$, then $g\in G_+\setminus \{0\}$.\ Also, $(G,G_+)$ is said to have the \emph{Riesz interpolation property} if for any $g_1,g_2,h_1,h_2\in G$ such that $g_i\leq h_j$ for any $i,j=1,2$ there exists $z\in G$ such that $g_i\leq z\leq h_j$ for any $i,j=1,2$.\ If all inequalities are replaced by strict inequalities, then $(G,G_+)$ is said to satisfy the \emph{strong Riesz interpolation property}. An ordered abelian group $(G,G_+)$ is said to have the \emph{Riesz decomposition property} if for any $g,h_1,h_2 \in G_+$ such that $g \leq  h_1 + h_2$, there exist $g_1,g_2 \in G_+$ such
that $g = g_1 + g_2$ and $g_j \leq h_j$ for each $j=1,2$.

Recall also that a countable ordered abelian group $(G, G_+)$ is called a \emph{dimension group} if it is isomorphic to the inductive limit of a sequence

\[
\begin{tikzcd}
\mathbb{Z}^{r_1} \ar{r}{\alpha_1} & \mathbb{Z}^{r_2} \ar{r}{\alpha_2} & \mathbb{Z}^{r_3} \ar{r}{\alpha_3} & \ldots,
\end{tikzcd}
\] for some natural numbers $r_n$, some positive group homomorphisms $\alpha_n$, where $\mathbb{Z}^r$ is equipped with its standard ordering given by $$(\mathbb{Z}^r)_+=\{(x_1,x_2,\ldots, x_r) : x_j\geq 0\}.$$ Using a classical theorem of Effros, Handelman, and Shen (\cite[Theorem $2.2$]{DimGroup80}), a countable ordered abelian group is a dimension group if and only if it is \emph{unperforated} and has the \emph{Riesz interpolation property}.

Recall that $I$ is an \emph{order ideal} of $(G,G_+)$ if $I$ is a subgroup of $G$ such that $$I=\left(I\cap G_+\right)-\left(I\cap G_+\right)$$ and  $$\text{if}\ 0\leq y\leq x \ \text{in} \ G_+\ \text{and} \ x\in I,\ \text{then}\ y\in I.$$
Moreover, an ordered abelian group $(G,G_+)$ is called \emph{simple} if any non-zero positive element is an order unit.

In Section \ref{sect: ProofA}, we will often use some of the properties of ordered groups defined above in the case when the group $G$ is a priori graded. To avoid any confusion, we remark that the meaning of these properties is unchanged. In particular, $G=G_0\oplus G_1$ is a \emph{graded ordered group} if it is graded and there exists a distinguished subset $G_+$ such that $(G,G_+)$ is an ordered group in the usual sense.

We can now state the range-of-the-invariant result in \cite{InvRangeElliott}. The theorem below is folklore (see also \cite[Remark 2.5]{classif} stated in terms of the $KT_u$-invariant) and it is implicitly contained in \cite{InvRangeElliott} as all of the inductive limit constructions provided in \cite{InvRangeElliott} are simple and have finite nuclear dimension, and hence are $\Z$-stable. We will use this result in the proof of Theorem \ref{thm:ExistenceTracial}.

\begin{theorem}[\cite{InvRangeElliott}]\label{thm: InvRange}
Let $(G,G_+,v)$ be a simple, countable, ordered abelian group which is weakly unperforated and has distinguished order unit $v$, $H$ be a countable abelian group, and $X$ be a compact, metrisable Choquet simplex together with a weakly unperforated pairing $\rho:G\times X\to \mathbb{R}$ which determines $G_+$. Suppose also that the pairing $\rho$ induces a surjection from $X$ onto the positive homomorphisms on $(G,G_+)$ which send $v$ to $1$. Then there exists a simple, separable, unital, nuclear, $\Z$-stable $\C$-algebra $A$ satisfying the UCT such that $$\Ell(A)\cong (G,G_+,v, H, X, \rho).$$ 
\end{theorem}

\begin{proof}
Let $(G,G_+,v, H, X, \rho)$ be as in the statement of the theorem. Then there exists a simple separable stable nuclear $\C$-algebra $B$ such that $\Ell(B)\cong (G,G_+,H, \tilde{X},\tilde{\rho})$, where $\tilde{X}$ is the cone with base $X$ and $\tilde{\rho}$ is the unique extension of the pairing $\rho$ (\cite[Theorem 5.2.3.2]{InvRangeElliott}). Since the $\C$-algebra $B$ is constructed as an inductive limit, one can simply tensor with $\Z$ at each stage to assume that $B$ is $\Z$-stable, or notice that $B$ is simple and has finite nuclear dimension as each building block has finite nuclear dimension (\cite[Proposition 2.3]{nucdim}). Thus, $B$ is $\Z$-stable by \cite[Corollary 8.7]{T14}. Moreover, the UCT is preserved by inductive limits (see \cite[22.3.5 (e)]{blackadar}), so $B$ satisfies the UCT. 

If $p\in B$ is a projection such that $[p]_0=v$ in $K_0(B)$, then we take $A=pBp$. Since $B$ is simple, $p$ is a full projection, so $A$ is simple separable and unital (see \cite{Brown77}). Moreover, $A$ is nuclear, $\Z$-stable, and it satisfies the UCT by Proposition \ref{prop: StableIsom}. Finally, it is immediate to see that $\Ell(A)\cong (G,G_+,v, H, X, \rho).$
\end{proof}

\subsection{Elliott's classification of A$\mathbb{T}$-algebras of real rank zero}\label{sect: AT}

This subsection contains a summary of Elliott's classification of (possibly non-simple) A$\mathbb{T}$-algebras of real rank zero from \cite{elliottclassifRR0}. 

Let $A$ be a unital $\C$-algebra, $K_*(A)$ be the graded group $K_0(A)\oplus K_1(A)$, equipped with the order 
\begin{equation*}
K_*(A)_+=\left\{\begin{array}{c}
([p]_0,[u\oplus (1_n-p)]_1): p\in \mathcal{P}(M_n(A)) \ \text{for some}\ n\in\mathbb N,\\ u\in\mathcal{U}(pM_n(A)p)
\end{array}
\right\},
\end{equation*}
and $$D_*(A)=\{([p]_0,[u\oplus (1-p)]_1)\in K_*(A)_+\colon p\in A\}.$$ If $A$ is nonunital, then define $$K_*(A)_+=K_*(A)\cap K_*(\tilde{A})_+,$$ and $$D_*(A)=K_*(A)\cap D_*(\tilde{A}),$$ where $\tilde{A}$ is the minimal unitisation of $A$. If $A$ is an A$\mathbb{T}$-algebra of real rank zero, consider the pair $(K_*(A),K_*(A)_+)$. Since $A$ also has stable rank one, the pair $(K_*(A),K_*(A)_+)$ is a graded ordered group (\cite[Theorem 3.2]{elliottclassifRR0}). Then, the graded group $K_*(A)$ together with the \emph{graded dimension range} $D_*(A)$ determine the pair $(K_*(A),K_*(A)_+)$ and constitutes the invariant used by Elliott in his classification of A$\mathbb{T}$-algebras of real rank zero. If $A$ is simple, the invariant $(K_*(A),K_*(A)_+)$ reduces to the usual $K$-theory triple $(K_0(A),K_0(A)_+,K_1(A))$ appearing in the Elliott invariant of a nonunital $\C$-algebra.

We record the following classification of isomorphisms between A$\mathbb{T}$-algebras of real rank zero from \cite{elliottclassifRR0}. Note that the statement does not appear in \cite{elliottclassifRR0}, but, as observed in \cite[Remark 7.3]{elliottclassifRR0}, the proof of \cite[Theorem 7.1]{elliottclassifRR0} not only produces a classification of A$\mathbb{T}$-algebras of real rank zero, but of isomorphisms between them.

\begin{theorem}[{\cite{elliottclassifRR0}}]\label{thm: ElliottAT}
Let $A$ and $B$ be A$\mathbb{T}$-algebras of real rank zero.\ Then $A$ is isomorphic to $B$ if and only if there is a graded group isomorphism $\alpha:K_*(A)\to K_*(B)$ such that $\alpha(D_*(A))=D_*(B)$. Moreover, for each such $\alpha$, there exists an isomorphism $\phi:A\to B$ such that $K_*(\phi)=\alpha$.
\end{theorem}

We also recall the range-of-the-invariant result in \cite{elliottclassifRR0} for A$\mathbb{T}$-algebras of real rank zero.

\begin{theorem}[{\cite[Theorem 8.1]{elliottclassifRR0}}]
Let $G_* = G_0 \oplus G_1$ be a graded ordered group which is the inductive limit of a sequence $G_{*1} \to G_{*2} \to \ldots$ of graded ordered groups such that each $G_{*i} = G_{0i}\oplus G_{1i}$ is the direct sum of finitely many basic building blocks of the form $\mathbb{Z}\oplus \mathbb{Z}$ with the order determined by strict positivity in the first component. If $G_*$ has the Riesz interpolation property, then there exists a stable A$\mathbb{T}$-algebra of real rank zero $A$ such that $K_*(A)\cong G_*$ and $K_*(A)_+\cong (G_*)_+$. 
\end{theorem}

We record the following folklore result relating traces and positive homomorphisms on the $K_0$-group of A$\mathbb{T}$-algebras of real rank zero.

\begin{prop}\label{thm: TracesAT}
Let $B$ be a separable, stable A$\mathbb{T}$-algebra of real rank zero.\ Then the canonical map from the densely defined lower semicontinous traces on $B$ to positive group homomorphisms $K_0(B)\to \mathbb R$ is an affine homeomorphism.
\end{prop}

\begin{proof}
This follows from \cite[Theorem 12.3]{G96} (or \cite[Theorem III.1.3]{BH82}) since any lower semicontinuous extended quasitrace on an A$\mathbb{T}$-algebra is an extended trace (\cite[Theorem 6]{BW11}).
\end{proof}

\subsection{Finite Rokhlin dimension}

The Rokhlin dimension of a group action on a unital $\C$-algebra was introduced in \cite{HWZ15} as a tool to obtain finite nuclear dimension and hence $\Z$-stability of the corresponding crossed product.\ In \cite{HP15}, Hirshberg and Phillips extended finite Rokhlin dimension to actions on possibly nonunital $\C$-algebras.\ We will use the simplified definition from \cite{bhishan}.

\begin{defn}[{\cite[Definition 4.6]{bhishan}}]\label{defn: RokhDim}
An automorphism $\gamma$ of a $\C$-algebra $A$ is said to have \emph{Rokhlin dimension} $d$ if $d$ is the least nonnegative integer with the following property. For any finite set $F\subseteq A$, integer $p\geq 1$ and $\epsilon>0$, there are positive contractions $f^{(l)}_{0},\dots,f^{(l)}_{p-1}\in A$, $l\in\{0,1,\dots,d\}$, such that:
\begin{enumerate}[label=\textit{(\roman*)}]
\item $\|f^{(l)}_{k}f^{(l)}_{j}a\| < \epsilon$ for every $a\in F$, $l\in\{0,1,\dots,d\}$, $j\neq k\in\{0,1,\dots,p-1\}$;\label{item:R1}
\item $\left\|\left(\sum_{l=0}^d\sum_{k=0}^{p-1}f^{(l)}_{k}\right)a-a\right\| < \epsilon$ for every $a\in F$;\label{item:R2}
\item $\|[f^{(l)}_{j},a]\| < \epsilon$ for every $a\in F$, $l\in\{0,1,\dots,d\}$, $j\in\{0,1,\dots,p-1\}$;\label{item:R4}
\item $\left\|\left(\gamma(f^{(l)}_{j})-f^{(l)}_{j+1}\right)a\right\| < \epsilon$ for every $a\in F$, $l\in\{0,1,\dots,d\}$, $j\in\{0,1,\dots,p-1\}$, where $f^{(l)}_{p}:=f^{(l)}_{0}$.\label{item:R3}
\end{enumerate}
\end{defn}

We will use finite Rokhlin dimension as a tool to detect when densely defined lower semicontinuous traces on a crossed product by $\mathbb{Z}$ can be identified with densely defined lower semicontinuous invariant traces on the algebra.\ The next lemma essentially follows from \cite[Proposition 2.3]{HCLiao16}, as the proof also works in the nonunital and nonsimple setting.\ That the proof works in the nonunital setting was previously observed in \cite[Lemma 4.8]{bhishan}. However, the proof in \cite[Proposition 2.3]{HCLiao16}, only deals with bounded traces. To generalise the result for densely defined lower semicontinuous traces, we will further assume the existence of an approximate unit of projections.

\begin{lemma}[cf.{\cite[Proposition 2.3]{HCLiao16}}]\label{lemma: TracesCrossProd}
Let $A$ be a separable $\C$-algebra with an approximate unit of projections $(q_r)_{r\in\mathbb{N}}$ and $\gamma$ be an automorphism of $A$ which has finite Rokhlin dimension. Then the restriction map from densely defined lower semicontinuous traces on $A\rtimes_\gamma\mathbb{Z}$ to $\gamma$-invariant densely defined lower semicontinuous traces on $A$ is bijective.
\end{lemma}

\begin{proof}
Using the strategy in \cite[Lemma 3.4]{KT21}, we will adapt Liao's proof of \cite[Proposition 2.3]{HCLiao16}.\ Let $P:A\rtimes_\gamma\mathbb{Z}\to A$ be the canonical conditional expectation.\ Given a densely defined lower semicontinuous trace $\tau$ on $A\rtimes_\gamma\mathbb{Z}$, we will show that $\tau=\tau|_A\circ P$. Since $q_r$ is a projection and hence contained in the Pedersen ideal of $A$, we have that $\tau(q_r)<\infty$. Moreover, $$\tau(x)=\lim\limits_{r\to\infty}\tau(q_rxq_r)$$ and $$\tau(P(x))=\lim\limits_{r\to\infty}\tau(q_rP(x)q_r)$$ for any positive element $x$ in $A\rtimes_\gamma\mathbb{Z}$. Therefore, to show that $\tau=\tau|_A\circ P$, it suffices to show that $$\tau(q_rxq_r)=\tau(q_rP(x)q_r)$$ for any positive element $x$ in $A\rtimes_\gamma\mathbb{Z}$ and any $r\in\mathbb N$. Since $\tau$ is nonzero, we can assume that $\tau(q_r)>0$ for any $r\in\mathbb{N}$.

Let $u$ be the unitary in $\mathcal{M}(A\rtimes_\gamma\mathbb{Z})$ implementing the action. Since the map $x\mapsto \tau(q_rxq_r)$ is bounded on $A\rtimes_\gamma\mathbb{Z}$ with norm $\tau(q_r)$, it suffices to show that $\tau(q_rau^nq_r)=0$ for any $a\in A$, and any $r,n\in\mathbb{N}$.

Recall that the automorphism $\gamma$ has finite Rokhlin dimension, say $d\in\mathbb{N}$.\ Let $r,n\in \N$, $a\in A$ be a contraction, and $\epsilon > 0$.\ Let $p = 2n$ and $\epsilon' = \frac{\epsilon}{(6(d+1)n+1)\tau(q_r)}$.\ Using \ref{item:R2} of Definition \ref{defn: RokhDim}, it follows that there exist $\{ f_{k}^{(\ell)}:0\leq k\leq p-1, \;0\leq \ell \leq d  \}$ in $A$ such that 
\begin{equation}\label{item: Liao1}
\left\| \left(\sum_{\ell=0}^d \sum_{k=0}^{p-1} f_k^{(\ell)}\right)a - a \right\| < \epsilon'.
\end{equation} Combining \ref{item:R1}, \ref{item:R4}, and \ref{item:R3} of Definition \ref{defn: RokhDim}, we can further assume that
\begin{equation}\label{item: Liao2}
\|  { f_k^{(\ell)} }^{ \frac{1}{2} }  a \gamma^n(  { f_k^{(\ell)} }^{\frac{1}{2}}  )  \| < \epsilon'\;\;\;\; (\ell=0,1,...,d;\; k=0,...,p-1).
\end{equation} Moreover, by \ref{item:R4} of Definition \ref{defn: RokhDim}, we can ensure that 
\begin{equation}\label{eq: Comm1}
\|{f_k^{(l)}}^{\frac{1}{2}}q_r-q_r{f_k^{(l)}}^{\frac{1}{2}}\|<\epsilon'
\end{equation}for any $0\leq k\leq p-1$ and any $0\leq l\leq d$.

Then
\begin{equation}\label{eq: Estim1}
| \tau(q_rau^nq_r) | \stackrel{\eqref{item: Liao1}}{\leq} \left|  \sum_{\ell=0}^d\sum_{k=0}^{p-1} \tau(q_rf_k^{(\ell)}au^nq_r)  \right| + \tau(q_r)\epsilon'.
\end{equation} Fix $k,l$ and let $x\coloneqq {f_k^{(l)}}^{\frac{1}{2}}au^n$. Using that $\tau$ is a trace and that $x$ is a contraction, we get that
\begin{equation}
\begin{array}{rcl}
|\tau(q_rx(q_r{f_k^{(l)}}^{\frac{1}{2}}-{f_k^{(l)}}^{\frac{1}{2}}q_r))|&=& |\tau(q_rx(q_r{f_k^{(l)}}^{\frac{1}{2}}-{f_k^{(l)}}^{\frac{1}{2}}q_r)q_r)|\\ &\leq&\tau(q_r)\|x(q_r{f_k^{(l)}}^{\frac{1}{2}}-{f_k^{(l)}}^{\frac{1}{2}}q_r)\|\\  &\stackrel{\eqref{eq: Comm1}}{\leq}& \tau(q_r)\epsilon'.
\end{array}
\end{equation} Similarly, one has that 
\begin{align*}
|\tau({f_k^{(\ell)}}^{\frac{1}{2}}q_r{f_k^{(\ell)}}^{\frac{1}{2}}au^nq_r)-\tau(q_rf_k^{(\ell)}au^nq_r)| &=|\tau(q_r({f_k^{(\ell)}}^{\frac{1}{2}}q_r-q_r{f_k^{(\ell)}}^{\frac{1}{2}})xq_r)|\\&\leq \tau(q_r)\epsilon'.    
\end{align*}

Therefore, using the last two estimations above, we get that 
\begin{align*}
    &|\tau(q_r{f_k^{(\ell)}}^{\frac{1}{2}}  a u^n  {f_k^{(\ell)}}^{\frac{1}{2}}q_r)-\tau(q_rf_k^{(\ell)}au^nq_r)|\\&\leq |\tau(q_r{f_k^{(\ell)}}^{\frac{1}{2}}au^nq_r{f_k^{(\ell)}}^{\frac{1}{2}})-\tau(q_rf_k^{(\ell)}au^nq_r)|+\tau(q_r)\epsilon'\\ &= |\tau({f_k^{(\ell)}}^{\frac{1}{2}}q_r{f_k^{(\ell)}}^{\frac{1}{2}}au^nq_r)-\tau(q_rf_k^{(\ell)}au^nq_r)|+\tau(q_r)\epsilon'\\ &\leq 2\tau(q_r)\epsilon'.
\end{align*}

Thus, from \eqref{eq: Estim1} we now get that 
$$|\tau(q_rau^nq_r)|\leq \sum_{\ell=0}^d\sum_{k=0}^{p-1} \left|  \tau(q_r{f_k^{(\ell)}}^{\frac{1}{2}}  a u^n  {f_k^{(\ell)}}^{\frac{1}{2}}q_r)   \right| +(2(d+1)p+1)\tau(q_r)\epsilon'$$
\begin{equation*}
\begin{array}{rcl}
&=&\sum_{\ell=0}^d\sum_{k=0}^{p-1} \left|  \tau(q_r{f_k^{(\ell)}}^{\frac{1}{2}} a \gamma^n( {f_k^{(\ell)}}^{\frac{1}{2}}   ) u^n q_r)    \right| + (2(d+1)p+1)\tau(q_r)\epsilon' \\
&\stackrel{\eqref{item: Liao2}}{\leq}&((d+1)p + 2(d+1)p+1)\tau(q_r)\epsilon' \\
&=&(6(d+1)n+1)\tau(q_r)\epsilon' = \epsilon.
\end{array}
\end{equation*}
Hence, $\tau(q_rau^nq_r)=0$ for any $r,n\in\mathbb{N}$, which yields that $\tau(q_rxq_r)=\tau(q_rP(x)q_r)$ for any $r\in\mathbb{N}$ and any $x\in A\rtimes_\gamma\mathbb{Z}$. Therefore, the canonical restriction induces an injection from the space of densely defined lower semicontinuous traces on $A\rtimes_{\gamma}\mathbb{Z}$ into the space of $\gamma$-invariant densely defined lower semicontinuous traces on $A$. Since this restriction is also surjective, the conclusion follows.
\end{proof}

\section{Bundles on Kirchberg algebras}\label{sect: ProofA}

In this section, we will prove Theorem \ref{thm:ExistenceKirchberg}. The construction of the required flow follows the strategy in \cite[Section 5]{EST}, together with some modifications that allow us to remove the assumption of no torsion in the $K_1$-group. We briefly describe the construction for the convenience of the reader, but we will often refer to \cite[Section 5]{EST}.

\subsection{Realising Kirchberg algebras as corners of crossed products by $\mathbb{Z}$}\label{subsect: Kirch}

Throughout this section, let $A$ be a unital $\UCT$ Kirchberg algebra and $(S,\pi)$ be a proper simplex bundle such that $\pi^{-1}(0)=\emptyset$.\ The main step in proving Theorem \ref{thm:ExistenceKirchberg} is realising $A$ as a full corner of a crossed product $B\rtimes_\gamma\mathbb{Z}$, with $B$ being a stable, possibly nonsimple, A$\mathbb{T}$-algebra with real rank zero. Comparing with the construction in \cite[Theorem 5.1]{EST}, it is precisely constructing $B$ to be an A$\mathbb{T}$-algebra rather than an AF-algebra that allows for torsion in $K_1(A)$.\ We collect the differences to the proof of \cite[Theorem 5.1]{EST} in Remark \ref{rmk: DifferenceEST}.

We first need to construct suitable $K$-theory groups, together with group isomorphisms between them. We will then obtain $B$ and $\gamma$ by using Elliott's classification of morphisms between stable, possibly nonsimple A$\mathbb{T}$-algebras with real rank zero from \cite{elliottclassifRR0}.\ We proceed to build the $K$-theory of the desired A$\mathbb{T}$-algebra $B$.\ Since $A$ will be identified with a full corner of a crossed product of $B$ with $\mathbb{Z}$, the construction of $K_0(B)$ needs to recover $K_0(A)$ via the Pimsner-Voiculescu exact sequence.\ To recover the simplex bundle $(S,\pi)$, we will ensure that $K_0(B)$ contains a dense subset of $\mathcal{A}(S,\pi)$. We will now sketch the relevant parts of the construction in \cite[Theorem 5.1]{EST}.  If $S=\emptyset$, then we can take $\theta$ to be the trivial flow on $A$, so assume that $S\neq \emptyset$ and $\pi^{-1}(0)=\emptyset$.

\begin{lemma}[cf.{\cite[Theorem 5.1]{EST}}]\label{lemma: PropGKirchberg}
Let $(S,\pi)$ be a nonempty proper simplex bundle such that $\pi^{-1}(0)=\emptyset$. Then there exists a countable torsion free ordered group $(G,G_+)$ together with an automorphism $\alpha$ of $(G,G_+)$ such that $G=\mathbb{Q}G_0$, where $G_0$ is a countable subgroup of $\mathcal{A}(S,\pi)$, $\alpha(g)(x)=e^{-\pi(x)}g(x)$ for any $g\in G$ and $x\in S$, and the following conditions hold:
\begin{enumerate}[label=\textit{(\roman*)}]
    \item for any $n\in\mathbb{N}$, $\epsilon>0$, and $f\in \mathcal{A}(S,\pi)$ such that $\supp(f)\subseteq \pi^{-1}([-n,n])$, there exists $g\in G_0$ with $\supp(g)\subseteq \pi^{-1}([-n,n])$ and $$\sup\limits_{x\in S}|f(x)-g(x)|<\epsilon;$$\label{item: ConstrG2}
    \item $(G,G_+)$ is a dimension group satisfying the strong Riesz interpolation property;\label{item: ConstrG4}
    \item the maps $\alpha$ and $\id-\alpha$  are automorphisms of $G$, and $\alpha(G_+)=G_+$;\label{item: ConstrG3}
    \item if $I$ is an order ideal in $G$ such that $\alpha(I)=I$, then $I=\{0\}$ or $I=G$;\label{item: ConstrG5}
    \item if $\psi\colon G\to\mathbb{R}$ is a positive group homomorphism such that $\psi(1)=1$, where $1$ denotes the constant function $1$ on $S$ and $\beta\in\mathbb{R}$ satisfies $\psi\circ\alpha=e^{-\beta}\psi$, then there exists a unique $\omega\in\pi^{-1}(\beta)$ such that $\psi(g)=g(\omega)$ for all $g\in G$.\label{item: ConstrG6}
\end{enumerate}
\end{lemma}

\begin{proof}
The construction of the pair $(G,G_+)$ is the one provided in \cite[Theorem 5.1]{EST}. Since the topology on $S$ is second countable, we can choose a countable subgroup $G_0$ of $\mathcal{A}(S,\pi)$ satisfying \ref{item: ConstrG2} in the statement of the lemma.\ Moreover, we can ensure that
\begin{itemize}
    \item the function $$x\mapsto e^{n\pi(x)}(1-e^{-\pi(x)})^mf(x)$$ is in $G_0$ for any $f\in G_0$ and any $n,m\in\mathbb{Z}$;
    \item the functions $$x\mapsto (\chi_{(-\infty,0]}\circ\pi)(x)e^{n\pi(x)}(1-e^{-\pi(x)})^m$$ and $$x\mapsto (\chi_{[0,\infty)}\circ\pi)(x)e^{n\pi(x)}(1-e^{-\pi(x)})^m$$ are in $G_0$ for any $n,m\in\mathbb{Z}$.\ The functions $\chi_{(-\infty,0]}\circ\pi$ and $\chi_{[0,\infty)}\circ\pi$ are continuous on $S$ because $\pi^{-1}(0)=\emptyset$.
\end{itemize} Let $G=\mathbb{Q}G_0$ 
and $G_+=\{f\in G: f(x)>0, x\in S\}\cup\{0\}.$ 

That $(G,G_+)$ is a dimension group satisfying the strong Riesz interpolation property follows from \cite[Lemma 5.3]{EST}.\ Since the function $$x\mapsto e^{n\pi(x)}(1-e^{-\pi(x)})^mf(x)$$ is in $G_0$ for any $f\in G_0$ and any $n,m\in\mathbb{Z}$, we get that $\alpha$ is an automorphism of $G$ and $\alpha(G_+)=G_+$.\ Furthermore, note that $\pi^{-1}(0)=\emptyset$, so $\id-\alpha$ is also an automorphism of $G$.\ This proves \ref{item: ConstrG3}.\ Condition \ref{item: ConstrG5} follows from the proof of \cite[Lemma 5.5]{EST}, while the existence of an element $\omega\in\pi^{-1}(\beta)$ as in \ref{item: ConstrG6} is shown in \cite[Lemma 5.6]{EST} (note that both the group $G$ and the map $\psi$ appearing in the proof of \cite[Lemma 5.6]{EST} coincide with the '$G$' and '$\psi$' in the statement of the lemma). 

Suppose that there exists $\omega'\neq \omega$ in $\pi^{-1}(\beta)$ such that $g(\omega)=g(\omega')$ for all $g\in G$. Then, by \ref{item: ConstrG2} above, $h(\omega)=h(\omega')$ for any $h\in \mathcal{A}(S,\pi)$ with compact support. As $\omega,\omega'\in \pi^{-1}(\beta)$, there exists $f\in \Aff(\pi^{-1}(\beta))$ such that $f(\omega)\neq f(\omega')$. Therefore, there exist $n\in\mathbb{N}$ and $h\in\mathcal{A}(S,\pi)$ such that the support of $h$ is contained in $\pi^{-1}([-n,n])$ and $h(\omega)\neq h(\omega')$, which is a contradiction. Hence, $\omega$ is unique.
\end{proof}

The pair $(G,G_+)$ only contains information about the simplex $S$.\ To complete the construction, the key ingredient is the following proposition which follows as a direct application of \cite[Proposition 3.5]{SESRordam}. 

\begin{prop}\label{prop: SESRordam}
There exist countable, abelian, torsion free groups $H_0\neq \{0\}$ and $H_1$, automorphisms $\kappa_0$ and $\kappa_1$ of $H_0$ and $H_1$ respectively and homomorphisms $q_i:H_i\to K_i(A)$ for $i=0,1$ such that 
 \[
\begin{tikzcd}
0 \ar{r} & H_i \ar{r}{\id-\kappa_i} & H_i\ar{r}{q_i} & K_i(A)\ar{r} & 0
\end{tikzcd}
\] is exact for $i=0,1$.
\end{prop}

\begin{proof}
We apply \cite[Proposition 3.5]{SESRordam} to the pairs $(K_0(A),0)$ and $(K_1(A),0)$.\ Moreover, if $H_0=\{0\}$, which happens only when $K_0(A)=0$, we take $H_0=\mathbb{Q}$ and $\kappa_0(x)=2x$.
\end{proof}

\begin{lemma}\label{lemma: ConstrK0Kirchberg}
Set $\Gamma=H_0\oplus G$ and if $p:\Gamma\to G$ is the canonical projection, let $$\Gamma_+=\{x\in \Gamma: p(x)\in G_+\setminus\{0\}\}\cup\{0\}.$$ Then the following conditions hold:
\begin{enumerate}[label=\textit{(\roman*)}]
    \item the pair $(\Gamma,\Gamma_+)$ is a dimension group;\label{lemma: dimgroup}
    \item for any $g\in \Gamma_+$, there exist $h_1,h_2\in \Gamma_+$ such that $g=2h_1+3h_2$ i.e. $\Gamma_+$ is weakly divisible;\label{lemma:LargeDen}
    \item If $I$ is a nonzero order ideal of $\Gamma$, then $I=H_0\oplus p(I)$;\label{item: IdealGamma}
    \item the only order ideals $I$ in $\Gamma$ such that $(\kappa_0\oplus\alpha)(I)=I$ are $I=\{0\}$ and $I=\Gamma$;\label{lemma: InvIdeals}
    \item the ordered group $(\Gamma,\Gamma_+)$ has no nonzero liminary subquotients i.e. no subquotient has all its prime quotients isomorphic to $\mathbb{Z}$.\label{item: liminary}
\end{enumerate}
\end{lemma}

\begin{proof}
We first check \ref{lemma: dimgroup}. As $(G,G_+)$ is a dimension group satisfying the strong Riesz interpolation property (Lemma \ref{lemma: PropGKirchberg}\ref{item: ConstrG4}) and $H_0$ is nonzero and torsion free, $(\Gamma, \Gamma_+)$ is a dimension group by \cite[Lemma 3.2]{DimGroup80}. 

To check \ref{lemma:LargeDen}, it suffices to find $h\in \Gamma_+$ such that $2h\leq g\leq 3h$.\ We then let $h_1=3h-g$ and $h_2=g-2h$. Let $g\in \Gamma_+$ be nonzero. Since $G=\mathbb{Q}G$, there exists $h'\in G_+$ nonzero such that $\frac{7}{3}h'\leq p(g)\leq \frac{8}{3}h'$. Then, $p(g-2(0,h'))=p(g)-2h'>0$ and $p(3(0,h')-g)=3h'-p(g)>0$, so we can take $h=(0,h')$ to conclude that $2h\leq g\leq 3h$ in $\Gamma$. 

We now check \ref{item: IdealGamma}. Let $I$ be a nonzero order ideal of $\Gamma$.\ Since $I=\left(I\cap \Gamma_+\right)-\left(I\cap\Gamma_+\right)$, there exist $h'\in H_0$ and $g\in G_+\setminus\{0\}$ such that $(h',g)\in I$.\ Then, for any $h\in H_0$, one has that $$0<(h+h',g)<2(h',g)$$ in $\Gamma$.\ Since $I$ is an order ideal, it follows that $(h+h',g)\in I$.\ Combining this with the fact that $(h',g)\in I$, yields that $(h,0)\in I$.\ Hence $H_0\oplus 0\subseteq I$.\ If $g\in p(I)$, there exists $h'\in H_0$ such that $(h',g)\in I$.\ If $h\in H_0$, then $(h+h',0)\in I$, so $(h,g)\in I$.\ Hence, we get that $I=H_0\oplus p(I)$.

The proof of \ref{lemma: InvIdeals} is contained in \cite[Lemma 5.5]{EST}, but we will include the details for the convenience of the reader.\ Let $I$ be a nonzero order ideal of $\Gamma$ such that $(\kappa_0\oplus\alpha)(I)=I$.\ By condition \ref{item: IdealGamma}, it follows that $I=H_0\oplus p(I)$.\ Moreover, $p(I)$ is an order ideal in $G$ such that $\alpha(p(I))=p(I)$.\ Condition \ref{item: ConstrG5} of Lemma \ref{lemma: PropGKirchberg} yields that $G=p(I)$.\ Thus, $I=\Gamma$.

We prove \ref{item: liminary} by contradiction. Let $J\subseteq I$ be order ideals of $\Gamma$ such that $I/J$ is nonzero and liminary in the category of ordered groups i.e. any prime quotient of $I/J$ is isomorphic to $\mathbb{Z}$. Then, by \ref{item: IdealGamma} above, $I=H_0\oplus p(I)$. Suppose that there exists a surjection $q:(H_0\oplus p(I))/J\to\mathbb{Z}$ and let $(h,i)$ be positive in $H_0\oplus p(I)$ such that $q(h,i)=1$.\ As $p(I)\subseteq G$ is an order ideal, it follows that $\frac{1}{2}i\in p(I)$. If $(h,\frac{1}{2}i)\in J$, then $(0,\frac{1}{2}i)\in J$ by \ref{item: IdealGamma}. Therefore, $(h,i)\in J$, which is a contradiction, so $(h,\frac{1}{2}i)\notin J$. Hence, one gets that $$0\leq q(h,\frac{1}{2}i)\leq 1.$$ If $q(h,\frac{1}{2}i)=0$, then $q(4h,2i)=0$. But $(h,i)\leq (4h,2i)$, so $1=q(h,i)\leq 0$, which is a contradiction. If $q(h,\frac{1}{2}i)=1$, then $q(0,\frac{1}{2}i)=0$, which gives that $q(0,2i)=0$. But $(h,i)\leq (0,2i)$, so $1=q(h,i)\leq 0$, which is a contradiction. Thus, the subquotient $I/J$ is not liminary.
\end{proof}

\begin{lemma}\label{lemma: BStates}
Choose $w\in H_0$ such that $q_0(w)=[1_A]_0$.\ Let $v=(w,1)\in \Gamma_+$, where $1$ stands for the constant function $1$ on $S$.\ Let $\varphi: \Gamma\to\mathbb{R}$ be a positive homomorphism and $\beta\in\mathbb{R}$ such that $\varphi(v)=1$ and $\varphi\circ(\kappa_0\oplus\alpha)=e^{-\beta}\varphi$.\ Then, there is a unique $\omega\in \pi^{-1}(\beta)$ such that $\varphi(h,g)=g(\omega)$ for all $(h,g)\in \Gamma$.\ In particular, there are no positive homomorphisms on $\Gamma$ sending $v$ to $1$ that are invariant under the automorphism $\kappa_0\oplus\alpha$.  
\end{lemma}

\begin{proof}
This follows as in \cite[Lemma 5.6]{EST}.\ Let $h\in H_0$ and $n\geq 1$. Then $n(h,0)+v\in \Gamma_+$ and $-n(h,0)+v\in \Gamma_+$, which yields that $n\varphi(h,0)+1\geq 0$ and $-n\varphi(h,0)+1\geq 0$ for all $n\geq 1$.\ Hence, $\varphi(h,0)=0$, which yields a positive homomorphism $\psi:G\to\mathbb{R}$ such that $\psi\circ p=\varphi$.\ Since $\varphi(v)=1$ and $p(v)=1$, it follows that $\psi(1)=1$.\ Moreover, $\varphi\circ(\kappa_0\oplus\alpha)=e^{-\beta}\varphi$ and $p\circ(\kappa_0\oplus\alpha)(g)=\alpha(g)$ for any $g\in G$, so $\psi\circ\alpha=e^{-\beta}\psi$.\ Thus, \ref{item: ConstrG6} of Lemma \ref{lemma: PropGKirchberg} shows that there is a unique $\omega\in \pi^{-1}(\beta)$ such that $\psi(g)=g(\omega)$ for all $g\in G$.\ Hence, $\varphi(h,g)=g(\omega)$ for all $(h,g)\in \Gamma$.\ The last statement follows by taking $\beta=0$, as $\pi^{-1}(0)=\emptyset$.
\end{proof}

We consider the triple $(\Gamma,\Gamma_+,H_1\oplus G)$ and we claim that it realises the $K$-theory of some stable A$\mathbb{T}$-algebra with real rank zero.\ This will follow from Elliott's classification of A$\mathbb{T}$-algebras with real rank zero from \cite{elliottclassifRR0}.\ We first need to impose a suitable order on the graded group $\Gamma\oplus (H_1\oplus G)$.\ We will do so using \cite[Theorem 4.28]{ElliottRR0II}.

\begin{lemma}\label{lemma: OrderAT}
Consider the graded group $\Lambda_*=\Gamma\oplus (H_1\oplus G)$ and define $$(\Lambda_*)_+=\{((h_0,g_0),(h_1,g_1))\in \Lambda_*\colon (h_0,g_0)\in \Gamma_+\setminus\{0\}, g_1\in I(g_0)\}\cup\{0\},$$ where $I(g_0)$ is the order ideal of $G$ generated by $g_0$.\ Then the following conditions hold:
\begin{enumerate}[label=\textit{(\roman*)}]
\item $(\Lambda_*,(\Lambda_*)_+)$ is a countable graded ordered group;\label{item: O0}
\item $(\Lambda_*,(\Lambda_*)_+)$ is unperforated;\label{item: O1}
\item $(\Lambda_*,(\Lambda_*)_+)$ has the Riesz decomposition property and the Riesz interpolation property;\label{item: O2}
\item $(\Lambda_*,(\Lambda_*)_+)$ is the inductive limit of a sequence $\Lambda_{*1}\to \Lambda_{*2} \to \ldots$ of graded ordered groups such that each $\Lambda_{*i} = \Lambda_{0i}\oplus \Lambda_{1i}$ is the direct sum of finitely many basic building blocks of the form $\mathbb Z \oplus\mathbb{Z}$ with the order determined by strict positivity in the first component.\label{item: O3}
\end{enumerate}
\end{lemma}

\begin{proof}
Conditions \ref{item: O0} and \ref{item: O2} follow as a direct application of \cite[Theorem 4.28]{ElliottRR0II}. It is worth noting that in \cite{ElliottRR0II}, the term ideal stands for what is called order ideal in this paper (see \cite[Remark 4.27]{ElliottRR0II}).

We first need to define a map from order ideals of $\Gamma$ to subgroups of $H_1\oplus G$.\ If $I$ is a nonzero order ideal of $\Gamma$, then $I=H_0\oplus p(I)$ by \ref{item: IdealGamma} of Lemma \ref{lemma: ConstrK0Kirchberg}.\ We then consider the map from order ideals of $\Gamma$ to subgroups of $H_1\oplus G$ given by $H_0\oplus p(I)\mapsto H_1\oplus p(I)$ which maps $0$ to $0$.\ This map also preserves inclusions, upward directed unions, and maps $\Gamma$ into $H_1\oplus G$.\ Condition \ref{item: O0} now follows from \cite[Theorem 4.28]{ElliottRR0II}.

We now prove \ref{item: O1}.\ Let $((h_0,g_0),(h_1,g_1))\in \Lambda_*$ and $n\in\mathbb{N}$ such that $n((h_0,g_0),(h_1,g_1))\in (\Lambda_*)_+$.\ Recall from Proposition \ref{prop: SESRordam} that $H_0$ and $H_1$ are torsion free.\ Then $\Lambda_*$ is torsion free, so we can assume that $(h_0,g_0)\in \Gamma\setminus\{0\}$.\ By \ref{lemma: dimgroup} of Lemma \ref{lemma: ConstrK0Kirchberg}, $\Gamma$ is a dimension group and hence unperforated, so $(h_0,g_0)\in\Gamma_+$.\ Moreover, since $n((h_0,g_0),(h_1,g_1))\in (\Lambda_*)_+$, we have that $ng_1$ is in the order ideal of $G$ generated by $ng_0$, say $I_0$.\ As $g_0$ is positive, it follows that $g_0\in I_0$, so $I_0$ is generated by $g_0$.\ Since $ng_1\in I_0$, there exists $k\in\mathbb{N}$ such that $$-kng_0\leq ng_1\leq kng_0$$ in $I_0$.\ Recall from Lemma \ref{lemma: PropGKirchberg} that $G=\mathbb{Q}G$. In particular, $G$ is divisible, so it follows that $-kg_0\leq g_1\leq kg_0$, which implies that $0\leq g_1+kg_0\leq 2kg_0$.\ This yields that $g_1+kg_0\in I_0$ and so $g_1\in I_0$.\ Thus $((h_0,g_0),(h_1,g_1))\in (\Lambda_*)_+$, which shows \ref{item: O1}.

By \ref{lemma: dimgroup} of Lemma \ref{lemma: ConstrK0Kirchberg}, $(\Gamma,\Gamma_+)$ is a dimension group, so it is weakly unperforated and satisfies the Riesz interpolation property, which coincides with the Riesz decomposition property (see \cite[Proposition 2.1]{goodearlbook}).\ The correspondence $H_0\oplus p(I)\mapsto H_1\oplus p(I)$ preserves finite sums and intersections.\ Moreover, no subquotient of $\Gamma$ is liminary by \ref{item: liminary} of Lemma \ref{lemma: ConstrK0Kirchberg}, so $\Lambda_*$ satisfies the Riesz decomposition property by \cite[Theorem 4.28]{ElliottRR0II}, thus proving \ref{item: O2}.

Condition \ref{item: O3} follows from \cite[Theorem 5.2]{Ell90} and the fact that $H_1\oplus G$ is torsion free.
\end{proof}

To apply classification results, we will need to construct an A$\mathbb{T}$-algebra which is $\Z$-stable.\ Building on Winter's techniques from \cite{Win12}, in \cite{ZStableNonSimple}, Robert and Tikuisis show that under some extra assumptions, finite nuclear dimension implies $\Z$-stability even in the nonsimple setting.\ We will use this result to prove that the A$\mathbb{T}$-algebra we obtain is $\Z$-stable.

\begin{lemma}\label{lemma: ConstrBKirch}
There exists a separable stable $A\mathbb{T}$-algebra $B$ with real rank zero such that the following conditions hold:
\begin{enumerate}[label=\textit{(\roman*)}]
\item $(K_0(B),K_0(B)_+,K_1(B))\cong (\Gamma,\Gamma_+,H_1\oplus G)$; \label{item: B1}

\item $D_*(B)\cong K_*(B)_+\cong (\Lambda_*)_+$;\label{item: B3}

\item $B$ is $\Z$-stable;\label{item: B2}
\end{enumerate}
\end{lemma}

\begin{proof}
Consider the graded group $\Lambda_*=\Gamma\oplus (H_1\oplus G)$ as in Lemma \ref{lemma: OrderAT}.\ Combining \ref{item: O2} and \ref{item: O3} of Lemma \ref{lemma: OrderAT} with \cite[Theorem 8.1]{elliottclassifRR0}, it follows that there exists a separable stable A$\mathbb{T}$-algebra $B$ with real rank zero such that \ref{item: B1} holds and $K_*(B)_+\cong (\Lambda_*)_+$. Since $B$ is stable, we can identify $K_*(B)_+$ with the graded dimension range $D_*(B)$ (see Section \ref{sect: AT}), so \ref{item: B3} is also satisfied.

We will prove that $B$ is $\Z$-stable. Since $B$ is an A$\mathbb{T}$-algebra, it follows that $B$ has finite decomposition rank (see \cite[Example 4.2]{decomprank}). We are going to use {\cite[Corollary 7.11]{ZStableNonSimple}}.\ Since $B$ has finite decomposition rank and real rank zero, no quotient of $B$ has a simple purely infinite ideal and the space of primitive ideals of $B$ has a basis of compact open sets (see the remark (a) after \cite[Corollary 7.11]{ZStableNonSimple}).\ Moreover, the Murray-von Neumann semigroup of $B$ is weakly divisible by \ref{lemma:LargeDen} of Lemma \ref{lemma: ConstrK0Kirchberg}, which yields that $B$ has no nonzero elementary ideal quotients (see \cite[Theorem 9.1]{NowhereScattered} together with the implication (1)$\implies$ (3) of \cite[Theorem B]{NowhereScattered}).\ Therefore, $B$ is $\Z$-stable by {\cite[Corollary 7.11]{ZStableNonSimple}}.
\end{proof}

In the next lemma, we will realise the automorphisms $\kappa_0\oplus\alpha$ and $\kappa_1\oplus\alpha$ as the $K$-theory of an automorphism $\gamma$ of $B$.\ Moreover, we choose $\gamma$ such that the crossed product $B\rtimes_\gamma\mathbb{Z}$ is simple using a variation of a result by Kishimoto from \cite{SimpleCrossedProdKishimoto}.\ Recall that an automorphism $\sigma$ on a $\C$-algebra $E$ is called \emph{properly outer} if for every nonzero $\sigma$-invariant closed two-sided ideal $I$ of $E$ and for every unitary $u$ in $\mathcal{M}(I)$ one has $\|\sigma|_I-\Ad(u)\|=2$.

\begin{lemma}\label{lemma: ConstrCrossProdKirch}
Let $B$ be the A$\mathbb{T}$-algebra obtained in Lemma \ref{lemma: ConstrBKirch}.\ Then there exists an automorphism $\gamma$ of $B$ such that the following conditions hold:
\begin{enumerate}[label=\textit{(\roman*)}]
\item $K_0(\gamma)=\kappa_0\oplus\alpha$ and $K_1(\gamma)=\kappa_1\oplus\alpha$;\label{item: CP1}
\item $B\rtimes_{\gamma}\mathbb{Z}$ is simple;\label{item: CP2}
\item $B\rtimes_{\gamma}\mathbb{Z}$ is $\Z$-stable;\label{item: CP3}
\item the canonical restriction induces a bijection between the cone of densely defined lower semicontinuous traces on $B\rtimes_{\gamma}\mathbb{Z}$ and the cone of $\gamma$-invariant densely defined lower semicontinuous traces on $B$.\label{item: CP4}
\end{enumerate}
\end{lemma}

\begin{proof}
Consider the automorphism of $\Gamma\oplus (H_1\oplus G)$ given by $(\kappa_0\oplus\alpha)\oplus(\kappa_1\oplus\alpha)$. Since $\alpha(G_+)=G_+$ by \ref{item: ConstrG3} of Lemma \ref{lemma: PropGKirchberg}, $(\kappa_0\oplus\alpha)\oplus(\kappa_1\oplus\alpha)$ preserves $(\Lambda_*)_+$. Therefore, we can apply Elliott's classification of automorphisms of A$\mathbb{T}$-algebras with real rank zero.\ Thus, by Theorem \ref{thm: ElliottAT}, there exists $\gamma'\in \Aut(B)$ such that $K_0(\gamma')=\kappa_0\oplus\alpha$ and $K_1(\gamma')=\kappa_1\oplus\alpha$.

We claim that we can replace $\gamma'$ by another automorphism $\gamma$ which induces the same maps in $K$-theory and the crossed product has finite nuclear dimension.\ Building on the work in \cite{RokhlinDim19}, we can take an automorphism $\gamma$ of $B$ with finite Rokhlin dimension such that $K_i(\gamma)=K_i(\gamma')$ for $i=0,1$ (\cite[Lemma 4.7]{bhishan}).\ Since $B$ has finite nuclear dimension and $\gamma$ has finite Rokhlin dimension, it follows that the crossed product $B\rtimes_{\gamma}\mathbb{Z}$ has finite nuclear dimension (\cite[Theorem 6.2]{RokhlinDim19} or \cite[Theorem 3.1]{HP15}).

We now show that $B\rtimes_{\gamma}\mathbb{Z}$ is simple.\ Since $\alpha^k$ is non-trivial for all $k\neq 0$, then $K_0(\gamma)^k$ is non-trivial, which implies that no non-trivial power of $\gamma$ is inner.\ Moreover, since $B$ is stable, the fact that $K_0(\gamma)^k$ is non-trivial, implies that there exists a projection $p_k\in B$ such that $\gamma^k(p_k)$ is not equivalent to $p_k$.\ Then, if $u$ is a unitary in $\mathcal{M}(B)$, $\gamma^k(p_k)$ is not equivalent to $up_ku^*$, so $$\|\gamma^k(p_k)-up_ku^*\|=1.$$ 

By \ref{lemma: InvIdeals} of Lemma \ref{lemma: ConstrK0Kirchberg}, the only order ideals $I$ in $\Gamma$ such that $(\kappa_0\oplus\alpha)(I)=I$ are $I=\{0\}$ and $I=\Gamma$.\ Thus, $\gamma^k$ is properly outer by the implication (iii)$\implies$(ii) in {\cite[Theorem 6.6]{SimpleCrossedProd}}. Moreover, by \ref{lemma: InvIdeals} of Lemma \ref{lemma: ConstrK0Kirchberg}, $K_0(B)$ has no non-trivial $K_0(\gamma)$-invariant ideals. Since $B$ is an A$\mathbb{T}$-algebra of real rank zero, this implies that $B$ has no non-trivial $\gamma$-invariant ideals (\cite[Proposition 1.5.3]{rordambook}), and hence the crossed product $B\rtimes_\gamma \mathbb{Z}$ is simple by {\cite[Theorem 7.2]{SimpleCrossedProd}}.\ Furthermore, $B\rtimes_\gamma \mathbb{Z}$ has finite nuclear dimension, so it is $\Z$-stable by \cite[Corollary 8.7]{T14}.

Note that $B$ has real rank zero, so it has an approximate unit of projections. Moreover, $\gamma$ has finite Rokhlin dimension. Hence, \ref{item: CP4} follows from Lemma \ref{lemma: TracesCrossProd}.
\end{proof}

We now have all the necessary ingredients to show that $A$ can be realised as a corner of a crossed product by the integers.

\begin{lemma}\label{lemma: CrossProdCorner}
Let $B$ the the stable A$\mathbb{T}$-algebra with real rank zero from Lemma \ref{lemma: ConstrBKirch}, and let $\gamma$ be the automorphism of $B$ from Lemma \ref{lemma: ConstrCrossProdKirch}. Then, there exists a projection $e\in B$ such that $e\left(B\rtimes_\gamma\mathbb{Z}\right)e\cong A$.
\end{lemma}

\begin{proof}
Recall that $v=(w,1)\in \Gamma_+$, where $w\in H_0$ is such that $q_0(w)=[1_A]_0$. Let $e$ be a projection in $B$ such that $[e]_0=v$ in $K_0(B)$. We will prove that $A\cong e\left(B\rtimes_\gamma\mathbb{Z}\right)e$ using the Kirchberg-Phillips classification theorem (\cite{kirchbergclass,phillipsclass}). First, we claim that $e\left(B\rtimes_\gamma\mathbb{Z}\right)e$ is a simple, separable, unital, nuclear, purely infinite $\C$-algebra satisfying the $\UCT$.

Note that $B\rtimes_\gamma\mathbb{Z}$ is simple by \ref{item: CP2} of Lemma \ref{lemma: ConstrCrossProdKirch}, so $e$ is a full projection. Moreover, $B$ is nuclear, satisfies the UCT, and $B\rtimes_\gamma\mathbb{Z}$ is $\Z$-stable by \ref{item: CP3} of Lemma \ref{lemma: ConstrCrossProdKirch}. Therefore, by Proposition \ref{prop: PermanenceProp}$, e\left(B\rtimes_\gamma\mathbb{Z}\right)e$ is a simple, separable, unital, nuclear, $\Z$-stable $\C$-algebra which satisfies the UCT. By {\cite[Corollary 5.1]{Rordam04}}, to show that $e\left(B\rtimes_\gamma\mathbb{Z}\right)e$ is purely infinite, it now suffices to prove that $e\left(B\rtimes_\gamma\mathbb{Z}\right)e$ has no quasitraces.\ Since $e\left(B\rtimes_\gamma\mathbb{Z}\right)e$ is unital and nuclear, any quasitrace is a trace by \cite{QuasiTraces}. Therefore, it suffices to show that $e\left(B\rtimes_\gamma\mathbb{Z}\right)e$ has no tracial states. If $\tau$ is a tracial state on $e\left(B\rtimes_\gamma\mathbb{Z}\right)e$, then it extends to a lower semicontinuous (possibly unbounded) trace $\tau'$ on $B\rtimes_\gamma\mathbb{Z}$ by \cite[Corollary 5.2]{nucdim}.\ By \ref{item: CP4} of Lemma \ref{lemma: ConstrCrossProdKirch}, the restriction of $\tau'$ to $B$ induces a positive group homomorphism $\tau'_*:K_0(B)\to \mathbb{R}$ such that $\tau'_*(v)=1$ and $\tau'_*\circ K_0(\gamma)=\tau'_*$.\ Recall from \ref{item: B1} of Lemma \ref{lemma: ConstrBKirch} and \ref{item: CP1} of Lemma \ref{lemma: ConstrCrossProdKirch} that $K_0(B)=\Gamma$ and $K_0(\gamma)=\kappa_0\oplus\alpha$.\ Thus, we obtain a contradiction with Lemma \ref{lemma: BStates}.\ Hence, $e\left(B\rtimes_\gamma\mathbb{Z}\right)e$ has no tracial states.\ Moreover, it is nuclear and $\Z$-stable, so it is purely infinite by {\cite[Corollary 5.1]{Rordam04}}.

Therefore, by the Kirchberg-Phillips' classification theorem, it suffices to check that $e\left(B\rtimes_\gamma\mathbb{Z}\right)e$ and $A$ have the same Elliott invariant. First note that $e\left(B\rtimes_\gamma\mathbb{Z}\right)e$ and $B\rtimes_\gamma\mathbb{Z}$ have the same $K$-groups. Applying the Pimsner-Voiculescu exact sequence (\cite[Theorem 2.4]{PV80}) to $B$ and the automorphism $\gamma$, we get the six-term exact sequence 
\[
\begin{tikzcd}
 H_0\oplus G\ar{r}{\id-K_0(\gamma)} & H_0\oplus G \ar{r} & K_0(B\rtimes_\gamma\mathbb{Z}) \ar{d} \\
 K_1(B\rtimes_\gamma\mathbb{Z}) \ar{u} & H_1\oplus G \ar{l}  & H_1\oplus G\ar{l}{\id-K_1(\gamma)}.
\end{tikzcd}
\] By \ref{item: CP1} of Lemma \ref{lemma: ConstrCrossProdKirch}, $\id_{H_1\oplus G}-K_1(\gamma)=(\id_{H_1}-\kappa_1)\oplus(\id_G-\alpha)$, which is injective by Proposition \ref{prop: SESRordam} and since $\id_G-\alpha$ is an automorphism of $G$ (\ref{item: ConstrG3} of Lemma \ref{lemma: PropGKirchberg}).\ Then, the downward map $K_0(B\rtimes_\gamma\mathbb{Z})\to H_1\oplus G$ is zero.\ This further implies that the map $H_0\oplus G\to K_0(B\rtimes_\gamma\mathbb{Z})$ is surjective, which yields that $$K_0(B\rtimes_\gamma\mathbb{Z})\cong(H_0\oplus G)/(\id_{H_0\oplus G}-K_0(\gamma))(H_0\oplus G).$$ By Lemma \ref{item: CP1} of \ref{lemma: ConstrCrossProdKirch}, $\id_{H_0\oplus G}-K_0(\gamma)=(\id_{H_0}-\kappa_0)\oplus(\id_G-\alpha)$.\ By \ref{item: ConstrG3} of Lemma \ref{lemma: PropGKirchberg}, $\id_G-\alpha$ is an automorphism of $G$, so $$K_0(B\rtimes_\gamma\mathbb{Z})\cong H_0/(\id_{H_0}-\kappa_0)(H_0)\cong K_0(A),$$where the last isomorphism is given by Proposition \ref{prop: SESRordam}. Note that $[e]_0=(w,1)$ in $K_0(B)$, where $w$ is mapped to $[1_A]_0$ by the isomorphism $H_0/(\id_{H_0}-\kappa_0)(H_0)\cong K_0(A)$. Therefore, the resulting isomorphism $K_0(e\left(B\rtimes_\gamma\mathbb{Z}\right)e)\to K_0(A)$ sends $[e]_0$ to $[1_A]_0$.

Further examining the Pimsner-Voiculescu exact sequence and using that $\id_{H_0\oplus G}-K_0(\gamma)$ is injective, it follows that the map $K_1(B\rtimes_\gamma\mathbb{Z})\to H_0\oplus G$ is zero.\ Therefore, we obtain the exact sequence 
 \[
\begin{tikzcd}
0 \ar{r}  & H_1\oplus G \ar{r}{\id-K_1(\gamma)} & H_1\oplus G\ar{r} & K_1(B\rtimes_\gamma\mathbb{Z}) \ar{r} & 0,
\end{tikzcd}
\] which yields that $K_1(B\rtimes_\gamma\mathbb{Z})\cong (H_1\oplus G)/(\id_{H_1\oplus G}-(\kappa_1\oplus\alpha))(H_1\oplus G)\cong H_1/(\id_{H_1\oplus G}-\kappa_1)(H_1)$, where the last identification follows as $\id_G-\alpha$ is an automorphism on $G$ (\ref{item: ConstrG3} of Lemma \ref{lemma: PropGKirchberg}).\ Hence $K_1(B\rtimes_\gamma\mathbb{Z})\cong K_1(A)$ by Proposition \ref{prop: SESRordam}.\ Thus, $e\left(B\rtimes_\gamma\mathbb{Z}\right)e\cong A$ by the Kirchberg-Phillips' classification theorem (see for example \cite[Theorem 8.4.1(iv)]{rordambook}).
\end{proof}

\subsection{KMS bundles on Kirchberg algebras}
In this subsection, we will finish the proof of Theorem \ref{thm:ExistenceKirchberg}. Recall from Lemmas \ref{lemma: ConstrBKirch} and \ref{lemma: ConstrCrossProdKirch} that $B$ is a stable $A\mathbb{T}$-algebra with real rank zero and $\gamma$ is an automorphism on $B$ with specified behaviour on $K$-theory such that $B\rtimes_\gamma\mathbb{Z}$ is $\Z$-stable.

Consider the dual action $\hat{\gamma}$ on $B\rtimes_\gamma\mathbb{Z}$ as a $2\pi$-periodic flow.\ Recall that $\hat{\gamma}_t(f)(x)=e^{-ixt}f(x)$ for any $t\in\mathbb{R}, f\in C_c(\mathbb Z, B)$ and $x\in\mathbb Z$, so that $\hat{\gamma}$ acts trivially on $B$. Since $e\in B$, $\hat{\gamma}$ restricts to an action on $e\left(B\rtimes_\gamma\mathbb{Z}\right)e$ which we denote by $\theta$.\ First, we need a result which will relate $\KMS$ states on $e\left(B\rtimes_\gamma\mathbb{Z}\right)e$ to positive homomorphisms on $K_0(B)$.\ The last statement of the following proposition is not explicitly contained in \cite[Lemma 4.1]{ET}. However, it is implied by the remark leading to \cite[Corollary 4.2]{ET}.

\begin{prop}\label{prop: tracesKMS}{\cite[Lemma 4.1]{ET}}
Let $D$ be a $\C$-algebra. Let $\rho$ be an automorphism of $D$ and let $q \in D$ a projection in $D$ which is full in $D\rtimes_\rho \mathbb{Z}$. Let $\hat{\rho}$ denote the restriction of the dual action on $D\rtimes_\rho \mathbb{Z}$ to $q(D\rtimes_\rho \mathbb{Z})q$ and consider it as a $2\pi$-periodic flow. If $P:D\rtimes_\rho\mathbb{Z}\to D$ is the canonical conditional expectation, then for each $\beta\in\mathbb{R}$, the map $\tau\mapsto \tau\circ P|_{q(D \rtimes_\rho \mathbb{Z})q}$ is an affine homeomorphism from the densely defined lower semicontinuous traces on $D$ satisfying $$\tau\circ\rho =e^{-\beta}\tau \quad \text{and} \quad \tau(q)=1,$$ onto the set of $\beta$-$\KMS$ states for the dual action $\hat{\rho}$ on $q(D \rtimes_\rho \mathbb{Z})q$. Moreover, the inverse is given by the map $\omega\mapsto\hat{\omega}|_D$, where $\hat{\omega}$ is a $\beta$-$\KMS$ weight for the dual action on $D\rtimes_\rho\mathbb{Z}$, which extends $\omega$.
\end{prop}

\begin{proof}
We will only comment on the proof of the last statement. Let $\omega$ be a $\beta$-$\KMS$ state for the dual action $\hat{\rho}$ on $q(D \rtimes_\rho \mathbb{Z})q$. Then, by \cite[Remark 3.3]{ExtendKMS}, there exists a unique $\beta$-$\KMS$ weight $\hat{\omega}$ for the dual action on $D\rtimes_\rho\mathbb{Z}$ which extends $\omega$.\ The result now follows by \cite[Lemma 3.1]{KT21}.
\end{proof}

\begin{proof}[Proof of Theorem \ref{thm:ExistenceKirchberg}] The proof follows the strategy in \cite[Lemma 5.8]{EST}. Recall that if $S=\emptyset$, then we take $\theta$ to be the trivial flow on $A$. Therefore, we assume that $S$ is nonempty and we claim that the $\KMS$ bundle $(S^\theta,\pi^\theta)$ of $\theta$ is isomorphic to $(S,\pi)$. This will finish the proof.    

Let $(\omega,\beta)\in S^\theta$. By \cite[Remark 3.3]{ExtendKMS}, there exists a unique $\beta$-$\KMS$ weight $\hat{\omega}$ for the dual action $\hat{\gamma}$ on $B\rtimes_\gamma\mathbb{Z}$ which extends $\omega$. Furthermore, $B\rtimes_\gamma\mathbb{Z}$ is simple by \ref{item: CP2} of Lemma \ref{lemma: ConstrCrossProdKirch}, so $e$ is full in $B\rtimes_\gamma\mathbb{Z}$.\ Then, by Proposition \ref{prop: tracesKMS}, $\hat{\omega}|_B$ is a densely defined lower semicontinuous trace on $B$ such that $\hat{\omega}|_B\circ\gamma=e^{-\beta}\hat{\omega}|_B$ and $\hat{\omega}|_B(e)=1$. Therefore, $(\hat{\omega}|_B)_*$ is a positive homomorphism on $K_0(B)$ such that $$(\hat{\omega}|_B)_*\circ K_0(\gamma)=e^{-\beta}(\hat{\omega}|_B)_* \ \text{and}\ (\hat{\omega}|_B)_*([e]_0)=1.$$ Then, by Lemma \ref{lemma: BStates}, there exists a unique $\omega'\in\pi^{-1}(\beta)$ such that 
\begin{equation}\label{eq: BundleHom}
(\hat{\omega}|_B)_*(h,g)=g(\omega')
\end{equation}
for all $(h,g)\in\Gamma\cong K_0(B)$. Define $\varphi:S^{\theta}\to S$ by $\varphi(\omega,\beta)=\omega'$ and note that $\pi\circ\varphi=\pi^\theta$ as $\omega'\in \pi^{-1}(\beta)$. Moreover, the restriction $\varphi:(\pi^\theta)^{-1}(\beta)\to \pi^{-1}(\beta)$ is affine by construction for any $\beta\in\mathbb{R}$.


We claim that $\varphi$ is surjective. Let $\mu\in S$ and define $\ev_{\mu}:K_0(B)\to\mathbb{R}$ by $\ev_{\mu}(h,g)=g(\mu)$. Recall that $K_0(\gamma)=\kappa_0\oplus\alpha$ by \ref{item: CP1} of Lemma \ref{lemma: ConstrCrossProdKirch} and $\alpha(g)(x)=e^{-\pi(x)}g(x)$ for any $g\in G$. Then, $\ev_{\mu}$ is a positive group homomorphism such that $\ev_{\mu}\circ K_0(\gamma)=e^{-\pi(\mu)}\ev_{\mu}$ and $\ev_{\mu}(v)=1$ as $v=(w,1)$. Then, by Proposition \ref{thm: TracesAT}, there exists a unique densely defined lower semicontinuous trace $\tau_\mu$ on $B$ such that $(\tau_\mu)_*=\ev_\mu$. By Proposition \ref{prop: tracesKMS}, there exists a $\pi(\mu)$-$\KMS$ state $\omega$ for $\theta$ such that $(\hat{\omega}|_B)_*=\ev_{\mu}$.\ Then $\varphi(\omega,\pi(\mu))=\mu$, so $\varphi$ is surjective.

We further check that $\varphi$ is injective. Let $(\omega_1,\beta_1),(\omega_2,\beta_2)\in S^\theta$ such that $\varphi(\omega_1,\beta_1)=\varphi(\omega_2,\beta_2)$. Then, $$\beta_1=\pi(\varphi(\omega_1,\beta_1))=\pi(\varphi(\omega_2,\beta_2))=\beta_2.$$ Moreover, by the definition of the map $\varphi$, $$(\hat{\omega}_1|_B)_*=(\hat{\omega}_2|_B)_*.$$ By Proposition \ref{prop: tracesKMS}, $\hat{\omega}_1|_B$ and $\hat{\omega}_2|_B$ are densely defined lower semicontinuous traces on $B$ such that $$(\hat{\omega}_1|_B)(e)=(\hat{\omega}_2|_B)(e)=1.$$ By Proposition \ref{thm: TracesAT}, it follows that $\hat{\omega}_1|_B=\hat{\omega}_2|_B.$ Then, Proposition \ref{prop: tracesKMS} yields that $\omega_1=\omega_2$, which shows that $\varphi$ is injective. 

If we show that $\varphi^{-1}:S\to S^\theta$ is continuous, then $\varphi$ is a homeomorphism by Lemma \ref{lemma: BundleHomeom}. Recall from Remark \ref{rmk: MetrisBundle} that both $S^\theta$ and $S$ are metrisable and let $\omega_n'$ be a sequence in $S$ which converges to $\omega'$. If $\varphi^{-1}(\omega_n')=(\omega_n,\beta_n)$ and $\varphi^{-1}(\omega')=(\omega,\beta)$, then \eqref{eq: BundleHom} yields that $(\hat{\omega}_n|_B)_*(h,g)$ converges to $(\hat{\omega}|_B)_*(h,g)$ for any $(h,g)\in \Gamma\cong K_0(B)$. Then, $\hat{\omega}_n|_B$ converges to $\hat{\omega}|_B$ by Proposition \ref{thm: TracesAT}, so $\omega_n$ converges to $\omega$ by Proposition \ref{prop: tracesKMS}. Moreover, $\beta_n=\pi(\omega_n')$, which converges to $\pi(\omega')=\beta$ by continuity of $\pi$. Thus, $(\omega_n,\beta_n)$ converges to $(\omega,\beta)$, so $\varphi^{-1}$ is continuous. Together with Lemma \ref{lemma: CrossProdCorner}, this yields the conclusion of Theorem \ref{thm:ExistenceKirchberg}.
\end{proof}


\begin{rmk}\label{rmk: DifferenceEST}
Note that in \cite[Section 5]{EST}, to realise the $K$-theory of a stable AF-algebra, the authors apply Proposition \ref{prop: SESRordam} to the pair $(K_0(A),K_1(A))$.\ Applying Proposition \ref{prop: SESRordam} to the pairs $(K_0(A),0)$ and $(K_1(A),0)$ instead of the pair $(K_0(A), K_1(A))$ is precisely what allows us to avoid assuming that $K_1(A)$ is torsion free.\ This idea has its roots in \cite[Theorem 3.6]{SESRordam}, where Rørdam showed that any pair of countable discrete abelian groups can be realised as the $K$-theory of an endomorphism crossed product of an $A\mathbb{T}$-algebra.\ Then, since we are working with an A$\mathbb{T}$-algebra, rather than an AF-algebra, different arguments are needed in Lemmas \ref{lemma: ConstrBKirch} and \ref{lemma: ConstrCrossProdKirch} to show that $B$ is $\Z$-stable and the crossed product is simple and $\Z$-stable.\ Moreover, the invariant for classifying possibly nonsimple A$\mathbb{T}$-algebras of real rank zero is more intricate than in the case of AF-algebras (see Section \ref{sect: AT}), so more detailed $K$-theory groups have to be constructed.
\end{rmk}

\section{Bundles on tracial classifiable C$^*$-algebras}\label{sect: ProofB}

In this section, we will prove Theorem \ref{thm:ExistenceTracial}. The construction of the required flow is in the spirit of \cite[Section 3]{EST}. Adapting the construction in \cite[Section 3]{EST} and using classification of unital $^*$-homomorphisms between classifiable $\C$-algebras from \cite{classif}, we will extend \cite[Theorem 3.14]{EST} to any simple, separable, unital, nuclear, stably finite, $\Z$-stable $\C$-algebra, with real rank zero, and which satisfies the $\UCT$.

\subsection{Realising tracial C$^*$-algebras as corners of crossed products by $\mathbb{Z}$}\label{subsect: Tracial}

We will follow a similar strategy as in Section \ref{sect: ProofA}, where we proved the corresponding result for Kirchberg algebras.\ Throughout this section, $A$ will be a simple, separable, unital, nuclear, stably finite, $\Z$-stable $\C$-algebra, with real rank zero, and which satisfies the $\UCT$. We also let $(S,\pi)$ be a compact simplex bundle such that $\pi^{-1}(0)\cong T(A)$. We will freely identify these simplices.\ The main step in proving Theorem \ref{thm:ExistenceTracial} is realising $A$ as a unital corner of a crossed product $(B\otimes\mathcal{K})\rtimes_\gamma\mathbb{Z}$, with $B$ being a simple, separable, unital, nuclear, stably finite, $\Z$-stable $\C$-algebra satisfying the $\UCT$.\ We first need to build the $K$-theory of the desired $\C$-algebra $B$.\ Let $\mathcal{A}_0(S,\pi)$ denote the set of elements $f\in \mathcal{A}(S,\pi)$ for which $f(T(A))=0$, and let $$\mathcal{A}_0(S,\pi)_+\coloneqq\{f\in\mathcal{A}_0(S,\pi)\colon f(x)>0 \ \text{for any}\ x\in \pi^{-1}(t) \ \forall \ t\neq 0\}.$$

\begin{lemma}\label{lemma: ConstrG0Tracial}
Let $(S,\pi)$ be a compact simplex bundle such that $\pi^{-1}(0)\cong T(A)$. Then there exists a countable subgroup $G_0$ of $\mathcal{A}_0(S,\pi)$ such that the following conditions hold:
\begin{enumerate}[label=\textit{(\roman*)}]
\item for any $f\in \mathcal{A}_0(S,\pi)$ and any $\epsilon>0$, there is an element $g\in G_0$ such that $$\sup\limits_{x\in S}|f(x)-g(x)|<\epsilon;$$ \label{item: TracialDens}

\item the map $\eta: g\mapsto e^{-\pi}g$ is an automorphism on $G_0$ and $\id-\eta$ is also an automorphism on $G_0$;\label{item: TracialAutCond}

\item For any $g\in G_0$, there exist $g_1,g_2\in G_0\cap \mathcal{A}_0(S,\pi)_+$ such that $g=g_1-g_2$.\label{item: OrdGrp}
\end{enumerate}
\end{lemma}

\begin{proof}
Since $S$ is compact and the topology on $S$ is second countable, we can choose a countable subgroup $G_0$ of $\mathcal{A}_0(S,\pi)$ such that for all $\epsilon>0$ and all $f\in \mathcal{A}_0(S,\pi)$, there is an element $g\in G_0$ such that $$\sup\limits_{x\in S}|f(x)-g(x)|<\epsilon.$$ Moreover, we can ensure that if $f\in G_0$, then $$x\mapsto e^{n\pi(x)}(1-e^{-\pi(x)})^mf(x)$$ is in $G_0$ for all $m,n\in\mathbb{Z}$.\ Then, a similar construction to the one in \cite[Section 4.2 pp 109-110]{ET} allows to choose $G_0$ such that both $\eta$ and $\id-\eta$ are surjective, hence giving \ref{item: TracialAutCond}. To check \ref{item: OrdGrp}, let $g\in G_0$ and $g_1,g_2\in\mathcal{A}_0(S,\pi)_+$ such that $g=g_1-g_2$. Using compactness of $S$, we can then ensure that $g_1,g_2\in G_0$.
\end{proof}

\begin{lemma}\label{lemma: ConstrKthTracial}
There exist a countable ordered abelian group $(G,G_+,v)$ with distinguished order unit $v$, a homomorphism $\hat{L}:G\to \mathcal{A}(S,\pi)$, and an automorphism $\alpha$ of $(G,G_+)$ such that the following conditions hold:
\begin{enumerate}[label=\textit{(\roman*)}]
\item $(G,G_+,v)$ is simple and weakly unperforated;\label{item: Kth1}

\item $\hat{L}(G)$ is uniformly dense in $\mathcal{A}(S,\pi)$;\label{item: Kth3}

\item the space of states on $G$, denoted by $S(G)$, is a metrisable Choquet simplex;\label{item: Kth2}

\item the homomorphism $\id-\alpha:G\to G$ is injective;\label{item: Kth4}

\item there is an isomorphism $\Sigma: G/(\id-\alpha)(G)\to K_0(A)$;\label{item: Kth6}

\item if $\beta\in\mathbb{R}$ and $\varphi$ is a state on $(G,G_+,v)$ with the property that $$\varphi\circ\alpha=e^{-\beta}\varphi,$$ then there is a unique $s\in \pi^{-1}(\beta)$ such that $\varphi=\ev_s\circ \hat{L}$.\footnote{The map $\ev_s$ denotes evaluation at $s\in \pi^{-1}(\beta)$.} In particular, the space of $\alpha$-invariant states on $G$ is affinely homeomorphic to $T(A)$.\label{item: Kth5}

\end{enumerate}
\end{lemma}

\begin{proof}
Set $$G=\Bigg(\bigoplus\limits_{\mathbb{Z}}K_0(A)\Bigg)\oplus G_0.$$ We will define an order on $G$ adapting the construction in \cite[Section 4.2]{ET}.\ The map $$r:\mathcal{A}(S,\pi)\to \Aff(T(A))$$ given by restriction is surjective ({\cite[Lemma 4.4(1)]{ET}}), so we can choose a positive linear map 
\begin{equation}\label{eq: Lmap}
L:\Aff(T(A))\to \mathcal{A}(S,\pi) \  \text{such that} \ r\circ L=\id \ \text{and} \ L(1)=1.
\end{equation}
If $\rho_A:K_0(A)\to \Aff(T(A))$ is the canonical pairing map of $A$, let us consider the homomorphism $\hat{L}:G\to \mathcal{A}(S,\pi)$ given by 
\begin{equation}\label{eq: LHatMap}
\hat{L}(\xi,g)(x)=g(x)+\sum\limits_{n\in\mathbb{Z}} L(\rho_A(\xi_n))(x) e^{n\pi(x)}
\end{equation}
for all $\xi=(\xi_n)_{n\in\mathbb{Z}}$, $g\in G_0$, and $x\in S$.\ Note that the map $\hat{L}$ is well-defined since $\xi$ has only finitely many nonzero entries.

Define 
\begin{equation}\label{eq: TracialOrder}
G_+=\{(\xi,g)\in G: \hat{L}(\xi,g)(x)>0,\ x\in S\}\cup \{0\}
\end{equation}
and set $v=(1^{(0)},0),$ where $1^{(0)}$ is the sequence with $[1_A]_0$ in the zero entry and zero elsewhere.

To show \ref{item: Kth1}, we will first show that the triple $(G,G_+,v)$ defines an ordered abelian group.\ The construction of $G_+$ in \eqref{eq: TracialOrder} yields that $G_+\cap(-G_+)=\{0\}$.\ We will check that $G=G_+-G_+$.\ Let $(\xi,g)\in G$. As $A$ is stably finite, $(K_0(A),K_0(A)_+)$ is an ordered abelian group (\cite[Proposition 6.3.3]{blackadar}), so there exist sequences $\xi^{(1)},\xi^{(2)}$ such that $\xi^{(i)}_n\in K_0(A)_+$ for any $i=1,2$ and $n\in\mathbb{N}$ and $\xi=\xi^{(1)}-\xi^{(2)}$.\ Moreover, by \ref{item: OrdGrp} of Lemma \ref{lemma: ConstrG0Tracial}, there exist $g_1,g_2\in G_0\cap\mathcal{A}_0(S,\pi)_+$ such that $g=g_1-g_2$. Thus, $(\xi,g)=(\xi^{(1)},g_1)-(\xi^{(2)},g_2)$, which shows that $G=G_+-G_+$.  

To show that $(G,G_+,v)$ is simple, we use the assumption that $S$ is compact.\ Let $(\xi,g)\in G_+\setminus\{0\}$.\ By definition, this yields that $\hat{L}(\xi,g)(x)>0$ for any $x\in S$.\ Since $S$ is compact, there exists $n\in\mathbb{N}$ such that $\hat{L}(\xi,g)(x)\geq \frac{1}{n}$ for any $x\in S$.\ This implies that $(\xi,g)$ is an order unit, so $(G,G_+,v)$ is a simple ordered abelian group.

To see that $(G,G_+)$ is weakly unperforated, let $n\in\mathbb{N}$ and $(\xi,g)\in G$ such that $n(\xi,g)\in G_+\setminus\{0\}$.\ Therefore, by linearity of $L$, it follows that $\hat{L}(\xi,g)(x)>0,\ x\in S$.\ Hence, $(\xi,g)\in G_+\setminus\{0\}$, so $(G,G_+)$ is weakly unperforated. 

We claim that \ref{item: Kth3} follows by combining \ref{item: TracialDens} of Lemma \ref{lemma: ConstrG0Tracial} and the fact that $A$ has real rank zero. Let $f\in\mathcal{A}(S,\pi)$ and $\epsilon>0$. Note $f|_{T(A)}\in \Aff(T(A))$. Since $A$ is unital, simple, exact, finite, $\Z$-stable, and has real rank zero, it follows that $\rho_A(K_0(A))$ is uniformly dense in $\Aff(T(A))$ (\cite[Theorem 7.2]{Rordam04}). Therefore, there exists $y\in K_0(A)$ such that 
\begin{equation*}\label{eq1}
    f(\tau)-\epsilon/2<\rho_A(y)(\tau)<f(\tau)+\epsilon/2
\end{equation*} for any $\tau\in T(A)$. As $T(A)$ is closed in $S$ and $S$ is compact, \cite[Lemma 2.3]{BEK86} shows that there exists $h\in\mathcal{A}(S,\pi)$ extending $\rho_A(y)$ such that 
\begin{equation}\label{eq2}
  f(x)-\epsilon/2< h(x) < f(x)+\epsilon/2  
\end{equation} for any $x\in S$. Then, $h-L(\rho_A(y))\in \mathcal{A}_0(S,\pi)$, so by \ref{item: TracialDens} of Lemma \ref{lemma: ConstrG0Tracial}, there exists $g\in G_0$ such that 
\begin{equation}\label{eq3}
  |h(x)-L(\rho_A(y))(x)-g(x)|<\epsilon/2
\end{equation} for any $x\in S$. Now consider the element $(y^{(0)},g)\in G$, where $y^{(0)}$ is the sequence which is constant $0$ apart from the zero entry which is equal to $y$. Then, 
\begin{equation*}
\begin{array}{rcl}
    \sup\limits_{x\in S}|\hat{L}(y^{(0)},g)(x)-f(x)|&=&\sup\limits_{x\in S}|g(x)+L(\rho_A(y))(x)-f(x)|\\ &\stackrel{\eqref{eq3}}{\leq}&\sup\limits_{x\in S}|f(x)-h(x)|+\epsilon/2\\ &\stackrel{\eqref{eq2}}{\leq}& \epsilon. 
\end{array}
\end{equation*} Thus, $\hat{L}(G)$ is uniformly dense in $\mathcal{A}(S,\pi)$.

We will now show \ref{item: Kth2}. This follows the strategy in \ref{item: Kth3}.\ Using \cite[Corollary II.3.11]{AlfsenBook} and Kadison's duality (\cite{Kad51}), it suffices to check that $G$ satisfies the strong Riesz interpolation property.\ Suppose that $(\xi_i,g_i)< (\eta_j,h_j)$ in $G$ for any $i,j=1,2$.\ Therefore, $\hat{L}(\xi_i,g_i)(x)<\hat{L}(\eta_j,h_j)(x)$ for any $x\in S$.\ Moreover, $\hat{L}(\xi_i,g_i)$ and $\hat{L}(\eta_j,h_j)$ restrict to affine functions on $T(A)$.

Since $A$ is unital, simple, exact, finite, $\Z$-stable, and has real rank zero, it follows that $\rho_A(K_0(A))$ is uniformly dense in $\Aff(T(A))$ (\cite[Theorem 7.2]{Rordam04}).\ Thus, $\rho_A(K_0(A))$ satisfies the strong Riesz interpolation property with respect to the strict ordering in $\Aff(T(A))$ (\cite[Lemma 3.1]{DimGroup80}), so there exists $y\in K_0(A)$ such that 
\begin{equation}
\hat{L}(\xi_i,g_i)(\tau)<\rho_A(y)(\tau)<\hat{L}(\eta_j,h_j)(\tau)    
\end{equation} for any $\tau\in T(A)$ and any $i,j=1,2$.\ As $T(A)$ is closed in $S$ and $S$ is compact, \cite[Lemma 2.3]{BEK86} shows that there exists $f\in\mathcal{A}(S,\pi)$ extending $\rho_A(y)$ such that \begin{equation}\label{eq:Tracial1}
\hat{L}(\xi_i,g_i)(x)<f(x)<\hat{L}(\eta_j,h_j)(x)    
\end{equation} for any $x\in S$ and any $i,j=1,2$.

Since $S$ is compact, we can let $\delta>0$ be smaller than $f(x)-\hat{L}(\xi_i,g_i)(x)$ and $\hat{L}(\eta_j,h_j)(x)-f(x)$ for any $x\in S$ and any $i,j=1,2$.\ By \ref{item: TracialDens} of Lemma \ref{lemma: ConstrG0Tracial}, there exists $g\in G_0$ such that
\begin{equation}\label{eq: Tracial2}
    |f(x)-L(\rho_A(y))(x)-g(x)|<\delta
\end{equation} for any $x\in S$.\ Now consider the element $(y^{(0)},g)\in G$, where $y^{(0)}$ is the sequence which is constant $0$ apart from the zero entry which is equal to $y$.\ We claim that $$\hat{L}(\xi_i,g_i)(x)<\hat{L}(y^{(0)},g)(x)<\hat{L}(\eta_j,h_j)(x)$$ for any $x\in S$.\ First note that for any $x\in S$ and $i=1,2$ we have that
\begin{equation*}
\begin{array}{rcl}
\hat{L}(y^{(0)},g)(x)-\hat{L}(\xi_i,g_i)(x) &=&g(x)+L(\rho_A(y))(x)-\hat{L}(\xi_i,g_i)(x)\\ &\stackrel{\eqref{eq: Tracial2}} {>}&f(x)-\delta-\hat{L}(\xi_i,g_i)(x)\\ &>&0,
\end{array}
\end{equation*} by the choice of $\delta$.\ A similar calculation shows that $\hat{L}(\eta_j,h_j)(x)>\hat{L}(y^{(0)},g)(x)$ for any $x\in S$ and $j=1,2$.\ Hence, by the definition of the order on $G$, we obtain that $$(\xi_i,g_i)<(y^{(0)},g)< (\eta_j,h_j)$$ for any $i,j=1,2$, which shows that $G$ satisfies the strong Riesz interpolation property.\ Thus, $S(G)$ is a Choquet simplex.\ Moreover, $S(G)$ is metrisable since $G$ is countable.

We will now define the automorphism $\alpha$ of $G$. Let $\alpha$ be the automorphism of $(G,G_+,v)$ given by $$\alpha(\xi,f)=(\sigma(\xi),e^{-\pi}f),$$ where $\sigma(\xi)=(\xi_{n+1})_{n\in\mathbb{Z}}$ is the left shift on $\bigoplus\limits_{\mathbb{Z}}K_0(A)$ and $e^{-\pi}$ denotes the function $x\mapsto e^{-\pi(x)}$. To check \ref{item: Kth4}, let $(\xi,g)\in G$ such that $(\id-\alpha)(\xi,g)=0$. Then $\sigma(\xi)=\xi$ and $e^{-\pi}g=g$. Since $\sigma$ is the shift on $\bigoplus_{\mathbb{Z}}K_0(A)$, it follows that $\xi=0$.\ Moreover, $g$ is supported away from $\pi^{-1}(0)$ by Lemma \ref{lemma: ConstrG0Tracial}, so $g=0$. Thus, $\id-\alpha$ is injective.

We will now check \ref{item: Kth6}. Consider the map $\Sigma_0:G\to K_0(A)$ given by $$\Sigma_0((\xi_n),f)=\sum\limits_{n\in\mathbb{Z}}\xi_n,$$ for any $(\xi_n)\in\bigoplus\limits_{\mathbb{Z}} K_0(A)$ and $f\in G_0$.\ We claim that the kernel of $\Sigma_0$ is $(\id-\alpha)(G)$.\ If $(\xi,g)\in G$, then $\Sigma_0((\id-\alpha)(\xi,g))=0$.\ Conversely, let $(\xi,g)\in G$ such that $\Sigma_0(\xi,g)=0$ and choose $N\in\mathbb{N}$ such that $\xi_n=0$ for all $|n|\geq N$.\ We follow the proof in \cite[Lemma 4.6]{KT21}.\ Set $z_n=0$ for $n\geq N$ and $$z_n=\xi_n+z_{n+1},\ n<N.$$ Then $$z_{-N}=\xi_{-N}+\sum_{i=1}^{2N}\xi_{-N+i}=0$$ since $\Sigma_0(\xi,g)=0$.\ This yields that $z_n=0$ for any $n\leq -N$.\ Moreover, if we set $z=(z_n)_{n\in\mathbb{Z}}$, we get that $(\id-\sigma)(z)=\xi$.\ By condition \ref{item: TracialAutCond} in Lemma \ref{lemma: ConstrG0Tracial}, it also follows that there exists $f\in G_0$ such that $g=(1-e^{-\pi})f$, so $(\id-\alpha)(z,f)=(\xi,g)$.\ Thus, the kernel of $\Sigma_0$ is $(\id-\alpha)(G)$, so $\Sigma_0$ induces an isomorphism $$\Sigma: G/(\id-\alpha)(G)\to K_0(A).$$ 

The proof of \ref{item: Kth5} follows as in \cite[Lemma 3.8]{EST}. Let $\beta\in\mathbb R$ and $\varphi$ be a state on $(G,G_+,v)$ such that $\varphi\circ\alpha=e^{-\beta}\varphi$.

The map $\hat{L}$ can be extended uniquely to a $\mathbb{Q}$-linear map on $\mathbb{Q}\otimes_\mathbb{Z}G$.\ For ease of notation, we will identify $\hat{L}$ with its extension $\hat{L}:\mathbb{Q}\otimes_\mathbb{Z}G\to  \mathcal{A}(S,\pi)$.\ Then, we can also extend $\varphi$ uniquely to a $\mathbb{Q}$-linear state $\hat{\varphi}$ of $\mathbb{Q}\otimes_\mathbb{Z}G$.\ Suppose that $\hat{L}(\xi,g)=0$ for some $(\xi,g)\in \mathbb{Q}\otimes_\mathbb{Z}G$.\ Then, $\frac{1}{n}v+(\xi,g),\frac{1}{n}v-(\xi,g)\in (\mathbb{Q}\otimes_\mathbb{Z}G)_+$, so $$-\frac{1}{n}\leq\hat{\varphi}(\xi,g)\leq \frac{1}{n}$$ for any $n\in\mathbb{N}$.\ Thus, $\hat{\varphi}(\xi,g)=0$, which implies that $\hat{\varphi}$ factors through $\hat{L}$. Indeed, there is a $\mathbb{Q}$-linear map $$\psi:\hat{L}(\mathbb{Q}\otimes_\mathbb{Z}G)\to\mathbb R, \ \text{such that}\ \psi\circ\hat{L}=\hat{\varphi}.$$

We claim that $\psi$ is continuous.\ Let $f\in \hat{L}(\mathbb{Q}\otimes_\mathbb{Z}G)$ and suppose by compactness of $S$ that $|f(x)|<\frac{n}{m}$ for any $x\in S$, for some $n,m\in\mathbb{N}$.\ Then, there exists $w\in\mathbb{Q}\otimes_\mathbb{Z}G$ such that $\hat{L}(w)=f$. Since $$-n=-n\hat{L}(v)(x)<m\hat{L}(w)(x)<n\hat{L}(v)(x)=n$$ for any $x\in S$, we get that $-nv<mw<nv$ in $\mathbb{Q}\otimes_\mathbb{Z}G$.\ Since $\psi\circ\hat{L}=\hat{\varphi}$ and $\hat{\varphi}(v)=1$, it follows that $|\psi(f)|=|\hat{\varphi}(w)|\leq\frac{n}{m}$. If there exists $g\in\hat{L}(\mathbb{Q}\otimes_\mathbb{Z}G)$ such that $|\psi(g)|>\sup\limits_{x\in S}|g(x)|$, then there exist $n,m\in\mathbb{N}$ such that $$|\psi(g)|>\frac{n}{m}>\sup\limits_{x\in S}|g(x)|.$$ But by the argument above, if $|g(x)|<\frac{n}{m}$ for any $x\in S$, then $|\psi(g)|\leq\frac{n}{m}$, so $\psi$ must be norm-contractive.

It follows from \ref{item: Kth3} above that $\hat{L}(\mathbb{Q}\otimes_\mathbb{Z}G)$ is uniformly dense in $\mathcal{A}(S,\pi)$.\ By continuity of $\psi$, we get that $\psi$ extends uniquely to a linear norm-contractive map $\psi:\mathcal{A}(S,\pi)\to\mathbb{R}$.\ The fact that there exists a unique $s\in \pi^{-1}(\beta)$ such that $\psi=\ev_s$ follows as in \cite[Lemma 3.8]{EST}.\ Hence, $\varphi=\ev_s\circ \hat{L}$ as required.
\end{proof}

\begin{rmk}
Note that the assumption that $A$ has real rank zero was instrumental in showing that $S(G)$ is a Choquet simplex.
\end{rmk}

As it was done in Proposition \ref{prop: SESRordam}, we now build a suitable $K_1$-group.

\begin{prop}\label{prop: SESRordamTracial}
There exist a countable, abelian, torsion free group $H$, an automorphism $\kappa$ of $H$ and a homomorphism $q:H\to K_1(A)$ such that 
 \[
\begin{tikzcd}
0 \ar{r} & H \ar{r}{\id-\kappa} & H\ar{r}{q} & K_1(A)\ar{r} & 0
\end{tikzcd}
\] is exact.
\end{prop}

\begin{proof}
We apply \cite[Proposition 3.5]{SESRordam} to the pair $(K_1(A),0)$.
\end{proof}

We can now construct a classifiable $\C$-algebra $B$. To show that we can identify $A$ with a corner of a crossed product of $B\otimes\mathcal{K}$, we will produce an automorphism of $B\otimes\mathcal{K}$ such that the crossed product is $\Z$-stable.\ The strategy is similar to the one in Section \ref{sect: ProofA}.\ We first show that we can choose an automorphism of $B\otimes\mathcal{K}$ which has finite Rokhlin dimension.\ Combining this with the fact that $B\otimes\mathcal{K}$ has finite nuclear dimension will yield that the obtained crossed product has finite nuclear dimension. 

\begin{lemma}\label{lemma: AutTracialCase}
Let $\rho_G:G\to \Aff(S(G))$ be the canonical map given by $\rho_G(g)(\phi)=\phi(g)$ for any $g\in G$ and $\phi\in S(G)$. There exist a simple, separable, unital, nuclear, $\Z$-stable $\C$-algebra $B$, satisfying the $\UCT$ and an automorphism $\gamma$ of $B\otimes\mathcal{K}$ such that the following conditions hold:
\begin{enumerate}[label=\textit{(\roman*)}]
\item $\Ell(B)\cong (G,G_+,v, H, S(G), \rho_G)$;\label{item: CP4Tra}
\item $K_0(\gamma)=\alpha$ and $K_1(\gamma)=\kappa$;\label{item: CP1Tra}  
\item $(B\otimes\mathcal{K})\rtimes_\gamma\mathbb{Z}$ is simple and $\Z$-stable;\label{item: CP2Tra}
\item the restriction map is a bijection from the densely defined lower semicontinuous traces on $(B\otimes\mathcal{K})\rtimes_\gamma\mathbb{Z}$ onto the $\gamma$-invariant densely defined lower semicontinuous traces on $B\otimes\mathcal{K}$.\label{item: CP3Tra}
\end{enumerate}
\end{lemma}

\begin{proof}
To construct a $\C$-algebra $B$ as in the statement of the lemma, we apply Elliott's range-of-the-invariant result (see Theorem \ref{thm: InvRange}). Lemma \ref{lemma: ConstrKthTracial} shows that $(G,G_+,v)$ is a simple, weakly unperforated, countable ordered abelian group and $S(G)$ is a metrisable Choquet simplex. Thus, by Theorem \ref{thm: InvRange}, there exists a simple, separable, unital, nuclear, $\Z$-stable $\C$-algebra $B$, satisfying the $\UCT$ such that $\Ell(B)\cong (G,G_+,v, H, S(G), \rho_G).$

We construct an automorphism $\gamma'$ which satisfies \ref{item: CP1Tra}. Set $p_1=1_B\otimes e_{11}$, where $e_{11}$ is a minimal projection in $\mathcal{K}$, and say that $\alpha([p_1]_0)=[p_2]_0$ for some projection $p_2\in B\otimes\mathcal{K}$.\ Then, essentially since any matrix amplification of $B\otimes\mathcal{K}$ is isomorphic to $B\otimes\mathcal{K}$, there exists a $^*$-isomorphism $\theta:B\otimes\mathcal{K}\to B\otimes\mathcal{K}\otimes\mathcal{K}$ such that $K_i(\theta)=K_i(\phi)$ for $i=0,1$ (see for example the argument in \cite[Corollary 6.2.11]{WO}), where $\phi:B\otimes\mathcal{K}\to B\otimes\mathcal{K}\otimes\mathcal{K}$ is given by $\phi(x)=x\otimes e_{11}$. By \cite{Brown77}, we get isometries $v_i\in \mathcal{M}(B\otimes\mathcal{K}\otimes\mathcal{K})$ such that $v_iv_i^*=p_i\otimes 1_{\mathcal{M}(\mathcal{K})}$ for $i=1,2$. Thus, the maps $$\eta_i:=\Ad(v_i)\circ \phi: B\otimes\mathcal{K}\to p_i(B\otimes\mathcal{K})p_i\otimes\mathcal{K}$$ induce isomorphisms at the level of $K_0$ and $K_1$ for any $i=1,2$.

We then set 
$$\alpha':=K_0(\eta_2)\circ\alpha\circ K_0(\eta_1)^{-1}$$ to be a unital ordered group isomorphism from $K_0(p_1(B\otimes\mathcal{K})p_1)$ to $K_0(p_2(B\otimes\mathcal{K})p_2)$. Moreover, set $$\kappa':=K_1(\eta_2)\circ\kappa\circ K_1(\eta_1)^{-1}$$ to be a group isomorphism from $K_1(p_1(B\otimes\mathcal{K})p_1)$ to $K_1(p_2(B\otimes\mathcal{K})p_2)$. 

By construction of $B$, the canonical map from tracial states on $B$ to states on $K_0(B)$ is an affine homeomorphism.\ Therefore, the tracial states on $p_i(B\otimes\mathcal{K})p_i$ are determined by states on $K_0(p_i(B\otimes\mathcal{K})p_i)$ for any $i=1,2$. Hence, it is immediate to see that there exists a unique linear isomorphism $\sigma:\Aff(T(p_1(B\otimes\mathcal{K})p_1))\to \Aff(T(p_2(B\otimes\mathcal{K})p_2))$, which is compatible with $\alpha'$ via the canonical pairing maps. Moreover, since $B$ is nuclear, $\Z$-stable, and satisfies the $\UCT$, and all these properties are preserved by stable isomorphisms (see Proposition \ref{prop: StableIsom}), so are $p_i(B\otimes\mathcal{K})p_i$ for $i=1,2$. Then, $p_1,p_2$ are full projections, so $p_i(B\otimes\mathcal{K})p_i$ is simple, separable, unital for $i=1,2$. Thus, by \cite[Corollary 9.5]{classif}, there exists a unital $^*$-isomorphism $\psi:p_1(B\otimes\mathcal{K})p_1\to p_2(B\otimes\mathcal{K})p_2$ such that $(K_0(\psi), K_1(\psi), \Aff(T(\psi)))= (\alpha', \kappa', \sigma)$. 

Set $\gamma'$ to be the following sequence of maps
 \[
\begin{tikzcd}
\gamma': B\otimes\mathcal{K} \ar{r}{\Ad(v_1)\circ\theta} & p_1(B\otimes\mathcal{K})p_1\otimes\mathcal{K} \ar{r}{\psi\otimes\id_{\mathcal{K}}} & p_2(B\otimes\mathcal{K})p_2\otimes\mathcal{K} \ar{r}{\theta^{-1}\circ \Ad(v_2^*)} & B\otimes\mathcal{K}.
\end{tikzcd}
\] By construction, we have that $K_0(\gamma')=\alpha$ and $K_1(\gamma')=\kappa$. 

To check condition \ref{item: CP2Tra}, we will first show that we can choose $\gamma$ such that the crossed product is simple and has finite nuclear dimension.\ Building on work in \cite{RokhlinDim19}, we can take an automorphism $\gamma$ of $B\otimes\mathcal{K}$ with finite Rokhlin dimension such that $K_i(\gamma)=K_i(\gamma')$ for $i=0,1$ (\cite[Lemma 4.7]{bhishan} or \cite[Theorem 3.4]{HWZ15}).\ Since $B\otimes\mathcal{K}$ is simple, nuclear, and $\Z$-stable, it has finite nuclear dimension (\cite[Theorem A]{CE20}).\ Moreover, $\gamma$ has finite Rokhlin dimension, so the crossed product $(B\otimes\mathcal{K})\rtimes_{\gamma}\mathbb{Z}$ has finite nuclear dimension (\cite[Theorem 5.2]{RokhlinDim19} or \cite[Theorem 3.1]{HP15}).\ Since $K_0(\gamma)^n\neq \id$ for all $n\neq 0$, no non-trivial power of $\gamma$ is inner. Moreover, $B\otimes\mathcal{K}$ is simple, so $(B\otimes\mathcal{K})\rtimes_\gamma\mathbb{Z}$ is simple by {\cite[Theorem 3.1]{SimpleCrossedProdKishimoto}}.\ Thus, we get \ref{item: CP2Tra} by \cite[Corollary 8.7]{T14}.

Since $B$ is unital and $\mathcal{K}$ has real rank zero, $B\otimes\mathcal{K}$ has an approximate unit of projections. Moreover, $\gamma$ has finite Rokhlin dimension, so \ref{item: CP3Tra} follows from Lemma \ref{lemma: TracesCrossProd}.
\end{proof}

As $B$ is unital and simple, the restriction map $\tau\mapsto \tau|_{B\otimes e_{11}}$ is a linear order-preserving isomorphism from the cone of densely defined lower semicontinuous traces on $B\otimes\mathcal{K}$ to the space of tracial states on $B$ (\cite[Proposition 4.7]{CuPed79}). But, by construction, the space of tracial states on $B$ is affinely homeomorphic to the space of states on $K_0(B)$, so we identify states on $K_0(B)$ with densely defined lower semicontinuous traces on $B\otimes\mathcal{K}$.\ We will now show that $A$ can be identified with a corner of the crossed product $(B\otimes\mathcal{K})\rtimes_\gamma\mathbb{Z}$.

\begin{lemma}\label{lemma: IsomCorner}
Let $B$ be the $\C$-algebra and $\gamma$ be the automorphism of $B\otimes\mathcal{K}$ both given by Lemma \ref{lemma: AutTracialCase} and $p=1\otimes e_{11}\in B\otimes\mathcal{K}$.\ Then $p\left((B\otimes\mathcal{K})\rtimes_\gamma\mathbb{Z}\right)p\cong A$.
\end{lemma}

\begin{proof}
We will use the classification theorem in \cite[Theorem 9.9]{classif}.\ First, we check that $p\left((B\otimes\mathcal{K})\rtimes_\gamma\mathbb{Z}\right)p$ is a simple, separable, unital, nuclear, $\Z$-stable $\C$-algebra satisfying the $\UCT$. Note that $(B\otimes\mathcal{K})\rtimes_\gamma\mathbb{Z}$ is simple by \ref{item: CP2Tra} of Lemma \ref{lemma: AutTracialCase}, so $p$ is a full projection. Moreover, $B\otimes\mathcal{K}$ is nuclear, satisfies the UCT, and $(B\otimes\mathcal{K})\rtimes_\gamma\mathbb{Z}$ is $\Z$-stable by \ref{item: CP2Tra} of Lemma \ref{lemma: AutTracialCase}. Therefore, by Proposition \ref{prop: PermanenceProp}$, p\left((B\otimes\mathcal{K})\rtimes_\gamma\mathbb{Z}\right)p$ is a simple, separable, unital, nuclear, $\Z$-stable $\C$-algebra which satisfies the UCT.

It remains to check that the Elliott invariant of $p\left((B\otimes\mathcal{K})\rtimes_\gamma\mathbb{Z}\right)p$ is isomorphic to the Elliott invariant of $A$.\ As discussed in Remark \ref{rmk: ElliottvsKTu}, we do not need to check that the positive cones in the $K_0$-groups coincide.\ Note first that the space of tracial states on $p\left((B\otimes\mathcal{K})\rtimes_\gamma\mathbb{Z}\right)p$ is linearly isomorphic to the space of densely defined lower semicontinuous traces on $(B\otimes\mathcal{K})\rtimes_\gamma\mathbb{Z}$ by \cite[Proposition 4.7]{CuPed79}. By \ref{item: CP3Tra} of Lemma \ref{lemma: AutTracialCase}, the latter is in a bijective correspondence to the space of $\gamma$-invariant densely defined lower semicontinuous traces on $B\otimes\mathcal{K}$.\ As observed previously, since $B$ is simple, unital and $S(K_0(B))\cong T(B)$, it follows that we can identify the space of $\gamma$-invariant densely defined lower semicontinuous traces on $B\otimes\mathcal{K}$ with the space of $K_0(\gamma)$-invariant states on $K_0(B)$. But the latter is homeomorphic to $\pi^{-1}(0)\cong T(A)$ by \ref{item: Kth5} of Lemma \ref{lemma: ConstrKthTracial}. Hence, $T\left(p\left((B\otimes\mathcal{K})\rtimes_\gamma\mathbb{Z}\right)p\right)\cong T(A)$.

To compute the $K$-groups of $p\left((B\otimes\mathcal{K})\rtimes_\gamma\mathbb{Z}\right)p$, we apply the Pimsner-Voiculescu exact sequence in \cite[Theorem 2.4]{PV80} to the $\C$-algebra $B\otimes\mathcal{K}$ and the automorphism $\gamma$. Identifying $K_i(B\otimes\mathcal{K})$ with $K_i(B)$ for $i=0,1$, we get that 
\[
\begin{tikzcd}
 K_0(B)\ar{r}{\id-K_0(\gamma)} & K_0(B) \ar{r} & K_0((B\otimes\mathcal{K})\rtimes_\gamma\mathbb{Z}) \ar{d} \\
 K_1((B\otimes\mathcal{K})\rtimes_\gamma\mathbb{Z}) \ar{u} & K_1(B) \ar{l}  & K_1(B)\ar{l}{\id-K_1(\gamma)}.
\end{tikzcd}
\] Recall from \ref{item: CP4Tra} of Lemma \ref{lemma: AutTracialCase} that $K_1(B)=H$ and from \ref{item: CP1Tra} of Lemma \ref{lemma: AutTracialCase} that $K_1(\gamma)=\kappa$. Then, Proposition \ref{prop: SESRordamTracial} gives that $\id-K_1(\gamma)$ is injective.\ Therefore, the map $K_0((B\otimes\mathcal{K})\rtimes_\gamma\mathbb{Z})\to K_1(B)$ is zero, which yields that the map $K_0(B)\to K_0((B\otimes\mathcal{K})\rtimes_\gamma\mathbb{Z})$ is surjective. Thus, 
\begin{equation}\label{eq: Kthtracial}
K_0((B\otimes\mathcal{K})\rtimes_\gamma\mathbb{Z})\cong K_0(B)/(\id-K_0(\gamma))(K_0(B)).
\end{equation}Since $K_0(B)=G$ by \ref{item: CP4Tra} of Lemma \ref{lemma: AutTracialCase} and $K_0(\gamma)=\alpha$ by \ref{item: CP1Tra} of Lemma \ref{lemma: AutTracialCase}, we get that $K_0((B\otimes\mathcal{K})\rtimes_\gamma\mathbb{Z})\cong K_0(A)$ by \ref{item: Kth6} of Lemma \ref{lemma: ConstrKthTracial}.\ This gives that, $$K_0\left(p\left((B\otimes\mathcal{K})\rtimes_\gamma\mathbb{Z}\right)p\right)\cong K_0(A).$$ Since $\Sigma([p]_0)=[1_A]_0$, it follows that the $K_0$-isomorphism is compatible with the position of the unit.

Combining \ref{item: CP1Tra} of Lemma \ref{lemma: AutTracialCase} and \ref{item: Kth4} of Lemma \ref{lemma: ConstrKthTracial} yields that the map  $\id-K_0(\gamma)$ is injective. Therefore, we obtain the short exact sequence 
 \[
\begin{tikzcd}
0 \ar{r}  & H \ar{r}{\id-K_1(\gamma)} & H\ar{r} & K_1((B\otimes\mathcal{K})\rtimes_\gamma\mathbb{Z}) \ar{r} & 0.
\end{tikzcd}
\] It follows that $K_1((B\otimes\mathcal{K})\rtimes_\gamma\mathbb{Z})\cong K_1(A)$ by Proposition \ref{prop: SESRordamTracial}.\ Hence, $$K_1\left(p\left((B\otimes\mathcal{K})\rtimes_\gamma\mathbb{Z}\right)p\right)\cong K_1(A).$$

To apply \cite[Theorem 9.9]{classif}, it remains to check that $p\left((B\otimes\mathcal{K})\rtimes_\gamma\mathbb{Z}\right)p$ and $A$ have the same pairing between $K$-theory and traces.\ Let $\tau$ be a tracial state on $A$ and $\tau_*$ be the induced state on $K_0(A)$.\ Recall that the homeomorphism from $T(A)$ onto $T(p\left((B\otimes\mathcal{K})\rtimes_\gamma\mathbb{Z}\right)p)$ is given by following sequence of mappings.\ One first uses \ref{item: Kth5} of Lemma \ref{lemma: ConstrKthTracial} to send $\tau$ to $\ev_{\tau}\circ \hat{L}$, which is a $K_0(\gamma)$-invariant state on $K_0(B)$.\ Then, \ref{item: CP4Tra} of Lemma \ref{lemma: AutTracialCase} yields that $\ev_{\tau}\circ \hat{L}$ corresponds to a unique $\gamma$-invariant tracial state on $B\otimes\mathcal{K}$.\ By \ref{item: CP3Tra} of Lemma \ref{lemma: AutTracialCase}, this extends uniquely to a densely defined lower semicontinuous trace on
$(B\otimes\mathcal{K})\rtimes_\gamma\mathbb{Z}$. As mentioned previously, the latter corresponds to a tracial state on
$p\left((B\otimes\mathcal{K})\rtimes_\gamma\mathbb{Z}\right)p$ (\cite[Proposition 4.7]{CuPed79}).

Therefore, it suffices to check that $$\tau_*\circ\Sigma_0 = \ev_{\tau}\circ \hat{L},$$ where $\Sigma_0:K_0(B)\to K_0(A)$ is the homomorphism in \ref{item: Kth6} of Lemma \ref{lemma: ConstrKthTracial} which induces the isomorphism $\Sigma$.
If $(\xi,g)\in K_0(B)$, then $$\tau_*\circ\Sigma_0(\xi,g)=\tau_*\left(\sum\limits_{n\in\mathbb{Z}}\xi_n\right)=\sum\limits_{n\in\mathbb{Z}}\rho_A(\xi_n)(\tau).$$ On the other hand, $$(\ev_{\tau}\circ\hat{L})(\xi,g)=g(\tau)+\sum\limits_{n\in\mathbb{Z}}L(\rho_A(\xi_n))(\tau)e^{n\pi(\tau)}.$$ Since $\pi(\tau)=0$ and $g\in G_0$ is supported away from $T(A)$ by Lemma \ref{lemma: ConstrG0Tracial}, it follows that $$(\ev_{\tau}\circ\hat{L})(\xi,g)=\sum\limits_{n\in\mathbb{Z}}L(\rho_A(\xi_n))(\tau).$$ Furthermore, recall from \eqref{eq: Lmap} that for any $f\in \Aff(T(A))$, $L(f)|_{T(A)}=f$. Thus, $$(\ev_{\tau}\circ\hat{L})(\xi,g)=\sum\limits_{n\in\mathbb{Z}}\rho_A(\xi_n)(\tau)=\tau_*\circ\Sigma_0(\xi,g).$$ Hence, as $A$ and $p\left((B\otimes\mathcal{K})\rtimes_\gamma\mathbb{Z}\right)p$ have the same Elliott invariant, they are isomorphic by \cite[Theorem 9.9]{classif}. 
\end{proof}

\subsection{KMS bundles on unital tracial classifiable C$^*$-algebras}
In this subsection, we will finish the proof of Theorem \ref{thm:ExistenceTracial}. Consider the dual action $\hat{\gamma}$ on $C=(B\otimes\mathcal{K})\rtimes_\gamma\mathbb{Z}$ as a $2\pi$-periodic flow.\ Recall that $\hat{\gamma}_t(f)(x)=e^{-ixt}f(x)$ for any $t\in\mathbb{R}, f\in C_c(\mathbb Z, B\otimes\mathcal{K})$ and $x\in\mathbb Z$.\ Since $p\in B\otimes\mathcal{K}$, $\hat{\gamma}$ restricts to an action on $pCp\cong A$ which we denote by $\theta$. We claim that the $\KMS$-bundle $(S^\theta,\pi^\theta)$ is isomorphic to $(S,\pi)$. We will crucially use \cite[Lemma 3.1]{KT21}, so we state it for the convenience of the reader.

\begin{lemma}{\cite[Lemma 3.1]{KT21}}\label{lemma: TracesTh}
Let $D$ be a $\C$-algebra and $\sigma \in \Aut(D)$ an automorphism of $D$. Let $\hat{\sigma}$ be the dual action on $D \rtimes_\sigma \mathbb Z$ considered as a $2\pi$-periodic flow. For $\beta \in\mathbb R$, the restriction map $\omega \mapsto \omega|_D$ is a bijection from the $\beta$-$\KMS$ weights for $\hat{\sigma}$ onto the densely defined lower semicontinuous traces $\tau$ on $D$ with the property that $\tau\circ\sigma = e^{-\beta}\tau$. The inverse is the map $\tau \mapsto \tau\circ P$, where $P : D\rtimes_\sigma\mathbb Z \to D$ is the canonical conditional expectation.    
\end{lemma}

\begin{proof}[Proof of Theorem \ref{thm:ExistenceTracial}]This follows the strategy in \cite[Lemma 3.13]{EST}. Let $(\omega,\beta)\in S^\theta$. By \cite[Remark 3.3]{ExtendKMS}, $\omega$ extends uniquely to a $\beta$-$\KMS$ weight $\hat{\omega}$ for $\hat{\gamma}$ on $C$. By Lemma \ref{lemma: TracesTh}, the restriction $\hat{\omega}|_{B\otimes\mathcal{K}}$ is a densely defined lower semicontinuous trace on $B\otimes\mathcal{K}$ such that $$\hat{\omega}|_{B\otimes\mathcal{K}}\circ\gamma=e^{-\beta}\hat{\omega}|_{B\otimes\mathcal{K}}.$$ Since $\hat{\omega}(1\otimes e_{11})=1$, \ref{item: Kth5} of Lemma \ref{lemma: ConstrKthTracial} gives a unique $s\in \pi^{-1}(\beta)$ such that 
\begin{equation}\label{eq: XiMap}
(\hat{\omega}|_{B\otimes\mathcal{K}})_*=\ev_s\circ \hat{L}.
\end{equation}
Then, we can define a map $\xi: S^\theta\to S$ by $\xi(\omega,\beta)=s$. By construction, $\pi\circ\xi=\pi^\theta$ and the restriction $\xi:(\pi^\theta)^{-1}(\beta)\to \pi^{-1}(\beta)$ is affine for any $\beta\in\mathbb{R}$.

We will first check that the map $\xi$ is injective. Let $(\omega_i,\beta_i)\in S^\theta$ such that $\xi(\omega_1,\beta_1)=\xi(\omega_2,\beta_2)=s$ for $i=1,2$. Since $\pi\circ\xi=\pi^\theta$, it follows that $\beta_1=\beta_2$. By construction of the map $\xi$, we have that $$\left(\hat{\omega}_1|_{B\otimes\mathcal{K}}\right)_*= \left(\hat{\omega}_2|_{B\otimes\mathcal{K}}\right)_*.$$ As the space of densely defined lower semicontinuous traces on $B\otimes\mathcal{K}$ is in a bijective correspondence to the space of states on $K_0(B)$, it follows that $\hat{\omega}_1|_{B\otimes\mathcal{K}}=\hat{\omega}_2|_{B\otimes\mathcal{K}}$, so $\hat{\omega}_1=\hat{\omega}_2$ by Lemma \ref{lemma: TracesTh}. Since $\hat{\omega}_i$ is an extension of $\omega_i$ for $i=1,2$, we get that $\omega_1=\omega_2$, so the map $\xi$ is indeed injective.

We will now check that $\xi$ is also surjective.\ Let $s\in S\cap\pi^{-1}(\beta)$ for some $\beta\in\mathbb R$.\ Then $\ev_s\circ \hat{L}$ is a state on $K_0(B)$ such that $$\ev_s\circ \hat{L}\circ K_0(\gamma)=e^{-\beta}\ev_s\circ \hat{L}.$$ By construction, states on $K_0(B)$ are uniquely induced by tracial states on $B$ (Lemma \ref{lemma: AutTracialCase}). These are in one-to-one correspondence with densely defined lower semicontinuous traces on $B\otimes\mathcal{K}$ (\cite[Proposition 4.7]{CuPed79}). Therefore, there exists a unique densely defined lower semicontinuous trace $\tau$ on $B\otimes\mathcal{K}$ such that $\tau_*=\ev_s\circ \hat{L}$ and $\tau_*\circ K_0(\gamma)=e^{-\beta}\tau_*$. Since the canonical map from densely defined lower semicontinuous traces on $B\otimes\mathcal{K}$ to states on $K_0(B)$ is a bijection (Lemma \ref{lemma: AutTracialCase} and \cite[Proposition 4.7]{CuPed79}), we get that $\tau\circ \gamma=e^{-\beta}\tau$. If $P:C\to B\otimes\mathcal{K}$ is the canonical conditional expectation, then the restriction $\tau\circ P|_{pCp}$ is a $\beta$-$\KMS$ state for $\theta$ by Lemma \ref{lemma: TracesTh}.\ By construction, $\xi(\tau\circ P|_{pCp},\beta)=s$, so $\xi$ is surjective.

If we show that $\xi^{-1}:S\to S^\theta$ is continuous, then $\xi$ is a homeomorphism by Lemma \ref{lemma: BundleHomeom}. Recall from Remark \ref{rmk: MetrisBundle} that both $S^\theta$ and $S$ are metrisable and let $s_n$ be a sequence in $S$ which converges to $s$. Then $\hat{L}(g)(s_n)$ converges to $\hat{L}(g)(s)$ for any $g\in G\cong K_0(B)$. If $\xi^{-1}(s_n)=(\omega_n,\beta_n)$ and $\xi^{-1}(s)=(\omega,\beta)$, then \eqref{eq: XiMap} yields that
$$\lim_{n\to\infty}(\hat{\omega}_n|_{B\otimes\mathcal{K}})_*(g)=(\hat{\omega}|_{B\otimes\mathcal{K}})_*(g)$$ for any $g\in K_0(B)$. Since the canonical map from the densely defined lower semicontinuous traces on $B\otimes\mathcal{K}$ to states on $K_0(B)$ is a bijection (follows again from Lemma \ref{lemma: AutTracialCase} and \cite[Proposition 4.7]{CuPed79}), $\hat{\omega}_n|_{B\otimes\mathcal{K}}$ converges pointwise to $\hat{\omega}|_{B\otimes\mathcal{K}}$. Thus, $\hat{\omega}_n$ converges pointwise to $\hat{\omega}$ by Lemma \ref{lemma: TracesTh}, so $\omega_n$ converges pointwise to $\omega$. Moreover, $\beta_n=\pi(s_n)$ converges to $\pi(s)=\beta$ by continuity of $\pi$, so $(\omega_n,\beta_n)$ converges to $(\omega,\beta)$. This shows that $\xi^{-1}$ is continuous. Hence, the $\KMS$-bundle $(S^\theta,\pi^\theta)$ is isomorphic to $(S,\pi)$. Combining this with Lemma \ref{lemma: IsomCorner} finishes the proof of Theorem \ref{thm:ExistenceTracial}.
\end{proof}

\bibliography{KMS}
\bibliographystyle{abbrv}
\end{document}